\documentclass[10pt]{article}
\usepackage{delarray} % LF: I have added this package.
\usepackage{amscd}
\usepackage{amssymb}
\usepackage{amsmath}
\usepackage{amsthm,amsxtra}
\usepackage{latexsym,amsmath,mathrsfs}
\usepackage[bookmarks,colorlinks]{hyperref}
\usepackage[all]{xy}
\usepackage{amsfonts}
\usepackage{color}
\usepackage{fancybox}
\usepackage{colordvi}
\usepackage{multicol}
\usepackage{textcomp}
\usepackage{stmaryrd}
\usepackage[dvips]{graphicx}
\usepackage{enumerate,xspace}
\usepackage{geometry}
\usepackage{tocbibind}
\renewcommand{\mod}{\hskip 6pt \mathrm{mod} \hskip 3pt}

\def\BZ{\mathbb{Z}}
\def\BQ{\mathbb{Q}}
\def\BR{\mathbb{R}}
\def\BA{\mathbb{A}}

\def\GL{\mathrm{GL}} \def\res{\mathrm{res}}

\def\fin{\mathrm{fin}} \def\sing{\mathrm{sing}}
\def\ord{\mathrm{ord}} \def\Sel{\mathrm{Sel}}
 
\def\ord{\mathrm{ord}} \def\Gal{\mathrm{Gal}}

    \theoremstyle{plain}

    \newtheorem{thm}{Theorem}[section] \newtheorem{cor}[thm]{Corollary}
    \newtheorem{lem}[thm]{Lemma} 
    \newtheorem{prop}[thm]{Proposition}
    \newtheorem {conj}[thm]{Conjecture}

    \theoremstyle{definition}

    \newtheorem{defn}[thm]{Definition}

    \theoremstyle{remark}
    \newtheorem {rem}[thm]{Remark}

    \numberwithin{equation}{section}

\newcommand{\wvec}[4]{{\scriptsize{\big ( \!\!
\begin{array}{cc} #1 \!\!\! & \!\!\! #2 \\ #3 \!\!\! & \!\!\! #4 \end{array} \!\! \big ) }}}

\begin{document}

\title{Iwasawa Theory of Hilbert modular forms for anticyclotomic extensions}
\author{Bingyong Xie \footnote{This paper is supported by
the National Natural Science Foundation of China (grant 11671137),
and supported in part by Science and Technology Commission of
Shanghai Municipality (no. 18dz2271000).}
 \\ \small Department of Mathematics, East China Normal University, Shanghai, China \\ \small byxie@math.ecnu.edu.cn} \date{}
\maketitle

\begin{abstract} Following Bertolini and Darmon's method, with ``Ihara's lemma'' among
other conditions Longo and Wang proved one divisibility of Iwasawa
main conjecture for Hilbert modular forms of weight $2$ and general
low even parallel weight in the anticyclotomic setting respectively.
In this paper, we remove the ``Ihara's lemma'' condition in their
results.
\end{abstract}

Key words: Selmer groups, $p$-adic $L$-functions, Iwasawa main
conjecture, anticyclotomic extensions.

MSC classification code: 11R23

\section*{Introduction}

Iwasawa theory studies the mysterious relation between pure
arithmetic objects and special values of complex $L$-functions. Its
precise statement is usually called ``main conjecture'' that
provides an equality between a quality measuring Selmer groups and a
$p$-adic $L$-function (interpolating the special values of a complex
$L$-function). Its proof is usually divided into two parts, one part
proving one divisibility by Ribet's method, the other proving the
converse divisibility by Euler systems.

In \cite{BD05} Bertolini and Darmon proved one divisibility of the
Iwasawa main conjecture for elliptic curve over $\BQ$ in the
anticyclotomic setting. Note that Bertolini and Darmon assumed a
$p$-isolated condition among other technical conditions. The
$p$-isolated condition was removed by Pollack and Weston \cite{PW}.
In \cite{CH15} Chida and Hsieh generalized this one divisibility to
low even weight elliptic modular forms. Their results were
generalized to the setting of Hilbert modular forms by Longo
\cite{Longo} for parallel weight $2$, and by Wang \cite{Wang} for
general low even parallel weights. There are other generalizations
obtained by Fouquet \cite{Fou} and Nekovar \cite{Nek}.

Their approach relies on a version of Ihara's Lemma. In the case of
elliptic modular forms, the needed Ihara's Lemma is Theorem 12 in
\cite{DT}. In the totally real case, \cite[Theorem 12]{DT} is
partially generalized by Jarvis \cite{Jav}. It seems that in the
unpublished paper \cite{Cheng} Ihara's Lemma was proved under the
conditions that the base totally real number field $F$ is
sufficiently small, i.e. $[F:\BQ]<p$, and that the level of the
Hilbert modular form in question is sufficiently large.  In their
recent preprint \cite{ManSho} Manning and Shotton proved Ihara's
lemma under the hypothesis that the image of $\bar{\rho}_f$ (a
modulo $p$ representation defined in our text) contains a subgroup
isomorphic to $\mathrm{SL}_2(\mathbb{F}_p)$. Thus under this strong
hypothesis Longo's and Wang's results are unconditional. In Section
\ref{sec:example} we will provide examples that do not satisfy
Manning and Shotton's hypothesis.

In this paper we remove the condition of Ihara's Lemma, and thus
obtain an unconditional result for all totally real number fields.
We need to persist technical conditions in \cite{Longo,Wang} other
than Ihara's Lemma. Instead of proving Ihara's Lemma, we take an
approach of avoiding it.

Let $F$ be a totally real number field, $\mathfrak{p}$ a place of
$F$ above $p$. Let $K$ be a totally imaginary quadratic extension of
$F$.  We form the anticyclotomic
$\BZ_p^{[F_\mathfrak{p}:\BQ_p]}$-extension $K_\infty$ of $K$. Put
$\Gamma=\mathrm{Gal}(K_\infty/K)$.

Let $f$ be a new Hilbert cusp form of parallel weight $k\geq 2$. Let
us write the level $\mathfrak{n}$ of $f$ in the form
$\mathfrak{n}=\mathfrak{n}^+\mathfrak{n}^-$, where $\mathfrak{n}^+$
is only divisible by prime ideals that split in $K$, and
$\mathfrak{n}^-$ is only divided by prime ideals that do not split
in $K$. We assume that $\mathfrak{n}^-$ is the product of different
prime ideals whose cardinal number has the same parity as $[F:\BQ]$.
This condition ensures that $f$ comes from a modular form on a
definite quaternion algebra with discriminant $\mathfrak{n}^-$. We
also assume $p\nmid \mathfrak{n}D_{K/F}$ and $f$ is ordinary at $p$.
Namely one of the two Hecke eigenvalues of $f$ at each place of $F$
above $p$ is a $p$-adic unit.
% We assume that $f$ comes from a modular form $\mathbf{f}$ on a totally
% definite quaternion algebra $B$ via the Jacquet-Langlands correspondence.

Let $\rho_f:G_F\rightarrow \mathrm{GL}_2(E_f)$ be the $p$-adic
Galois representation attached to $f$ (see \cite{Wiles, Tay} among
other references), where $E_f$ is the defining field of $\rho_f$.
Then $\det \rho_f= \epsilon^{k-1}$, where $\epsilon$ is the $p$-adic
cyclotomic character of $G_F=\mathrm{Gal}(\overline{F}/F)$. We
consider the self-dual twist of $\rho_f$, namely $\rho_f^*= \rho_f
\otimes \epsilon^{\frac{2-k}{2}}$. Let $V_f$ be the underlying
representation space for $\rho_f^*$. Fix a Galois stable lattice
$T_f$ of $V_f$, and put $A_f=V_f/T_f$.

Let $\mathrm{Sel}(K_\infty, A_f)$ be the minimal Selmer group of
$A_f$. Put $\Lambda=\mathcal{O}_f[[\Gamma]]$, where $\mathcal{O}_f$
is the ring of integers in $E_f$. Then $\mathrm{Sel}(K_\infty, A_f)$ and its Pontryagin dual
$\mathrm{Sel} (K_\infty, A_f)^\vee $ are $\Lambda$-modules.

On the other hand, one can attach to $f$ an anticyclotomic $p$-adic
$L$-function $L_p(K_\infty, f)\in\Lambda$ that interpolates the
special values $L(f/K,\chi,k/2)$ of the $L$-function of $f$ (where
$\chi$ runs over anticyclotomic characters).

\begin{conj} $($Iwasawa main conjecture$)$.  $\mathrm{Sel}(K_\infty, A_f)$ is a cofinitely
generated cotorsion $\Lambda$-module, and its characteristic ideal
$\mathrm{char}_\Lambda\mathrm{Sel} (K_\infty, A_f)^\vee \in \Lambda$
satisfies  $$ \mathrm{char}_\Lambda\mathrm{Sel} (K_\infty, A_f)^\vee
= (L_p(K_\infty, f)). $$
\end{conj}

Our main result is the following.

\begin{thm}\label{thm:main} Assume that $f$ satisfies the conditions $(\mathrm{CR}^+)$,
$(\mathrm{PO})$ and $(\mathfrak{n}^+\text{-}\mathrm{DT})$ given in
\cite{Wang}. Then $\mathrm{Sel}(K_\infty, A_f)$ is a cofinitely
generated cotorsion $\Lambda$-module, and
$$ \mathrm{char}_\Lambda\mathrm{Sel}
(K_\infty, A_f)^\vee \ | \  (L_p(K_\infty, f)). $$
\end{thm}

As applications of Theorem \ref{thm:main}, we have the following
consequences.

\begin{cor}\label{cor-a} Let $A$ be a modular elliptic curve $($or more generally a modular abelian variety of $\GL_2$-type$)$ over $F$.
Assume that $F_\mathfrak{p}=\mathbb{Q}_p$ and the modular form
attached to $A$ satisfies the assumption in Theorem \ref{thm:main}.
Then $A(K_\infty)$ is finitely generated.
\end{cor}

In \cite{Hun} Hung proved vanishing of the analytic $\mu$-invariant,
generalizing the result of Chida and Hsieh \cite{CH16}. Combining
Theorem \ref{thm:main} and Hung's result, we obtain the following

\begin{cor}\label{cor-b} Keep the assumption of Theorem \ref{thm:main}. Then the algebraic $\mu$-invariant of the $\Lambda$-module $\mathrm{Sel}(K_\infty, A_f)^\vee$ is zero.
\end{cor}

Corollary \ref{cor-a} and Corollary \ref{cor-b} were already
obtained by Longo \cite{Longo} and Wang \cite{Wang} respectively,
under the assumption of ``Ihara's Lemma''. \vskip 10pt

The strategy for the proof of Theorem \ref{thm:main} is using the
Euler system of Heegner points
$\{\kappa_{\mathscr{D}}(\mathfrak{l})_m\}_{\mathfrak{l}}$ to bound
the Selmer groups. In \cite{Wang} these Heegner points were shown to
satisfy two properties called  the First Reciprocity Law and the
Second Reciprocity Law. The Second Reciprocity Law needs ``Ihara's
Lemma''. Our input is to prove a weaker form of the Second
Reciprocity Law without ``Ihara's Lemma''. Our weaker version is
sufficient for us to run through Bertolini and Darmon's Euler system
argument to prove Theorem \ref{thm:main}. This is done in Section
\ref{sec:BD-argument}. See Proposition \ref{thm:first} and Corollary
\ref{cor:theta-2} for the precise statements of the First
Reciprocity Law and the weaker version of the Second Reciprocity
Law.

Both the original Second Reciprocity Law
\begin{equation}\label{eq:sec}
v_{\mathfrak{l}_2}(\kappa_{\mathscr{D}}(\mathfrak{l}_1)_m) =
v_{\mathfrak{l}_1}(\kappa_{\mathscr{D}}(\mathfrak{l}_2)_m)
\end{equation} (with $\mathfrak{l}_1$ and $\mathfrak{l}_2$ being different $n$-admissible primes) and our weaker version, are based on an analysis of the specialization modulo
$\omega$ (=\:the uniformizing element of $\mathcal{O}_f$) of Heegner
points to supersingular points. Starting from an $(N,n)$-admissible
form $g$ (Definition \ref{defn:adm}), using this specialization we
obtain a map $$ \gamma: B''^\times\backslash
\widehat{B}''^{\times}/Y \mathfrak{U}'' \rightarrow
\mathcal{O}_{f,n},
$$ (see the proof of Proposition \ref{thm:second} for the meanings of the notations),
which is expected to define a new $(N,n)$-admissible form. In
\cite{Wang}, $N$ is taken to coincide with $n$. Our
$(N,n)$-admissible form is called $n$-admissible form in loc. cit.

In \cite{Wang} Ihara's Lemma was used to show that $\gamma$ is
nonzero modulo $\omega$, i.e. the order of $\gamma$ is zero, so that
$\gamma$ really defines an $(N,n)$-admissible form denoted by $g''$
in our text. Wang \cite{Wang} shows that
\begin{equation}\label{eq:new-forms}
v_{\mathfrak{l}_2}(\kappa_{\mathscr{D}}(\mathfrak{l}_1)_m)
=\theta_m(g''). \end{equation} With $\mathfrak{l}_1$ and
$\mathfrak{l}_2$ exchanged one obtains another $n$-admissible form
$h''$ such that
$$v_{\mathfrak{l}_1}(\kappa_{\mathscr{D}}(\mathfrak{l}_2)_m)
=\theta_m(h'').
$$ Then the multiplicity one result $g''=h''$ yields
$(\ref{eq:sec})$. Both (\ref{eq:sec}) and (\ref{eq:new-forms}) are
needed in Bertolini and Darmon's (inductive) Euler system argument.

In our approach, we deal with (\ref{eq:sec}) and
(\ref{eq:new-forms}) separately.

We use the global Tate pairing to prove a weaker version of
(\ref{eq:sec}). We show that
$v_{\mathfrak{l}_2}(\kappa_{\mathscr{D}}(\mathfrak{l}_1))$ and
$v_{\mathfrak{l}_1}(\kappa_{\mathscr{D}}(\mathfrak{l}_2))$ coincide
with each other after multiplying $\theta(g)$. Indeed, by relations
like
$$\sum_v\langle \kappa_{\mathscr{D}}(\mathfrak{l}_1)_m, \kappa_{\mathscr{D}}(\mathfrak{l}_2)_m
\rangle_{v}=0$$ provided by the global Tate pairing between
$H^1(K_m,T_{f,n})$ and itself (noting that $T_{f,n}\cong A_{f,n}$)
we obtain
$$\theta(g) \cdot v_{\mathfrak{l}_2}(\kappa_{\mathscr{D}}(\mathfrak{l}_1))= \theta(g) \cdot v_{\mathfrak{l}_1}(\kappa_{\mathscr{D}}(\mathfrak{l}_2)) $$
up to multiplication by a unit in $\mathcal{O}_{f,n}[[\Gamma]]$. If
a homomorphism $\varphi: \mathcal{O}_{f,n}[[\Gamma]]\rightarrow
\mathcal{O}_{f,n}$ satisfies \begin{equation}\label{eq:cond} \varphi
\Big(\theta(g) \cdot
v_{\mathfrak{l}_2}(\kappa_{\mathscr{D}}(\mathfrak{l}_1))\Big)\neq 0,
\end{equation} in $\mathcal{O}_{f,n}$,  then we have
\begin{equation} \label{eq;apply}
\varphi\Big(v_{\mathfrak{l}_2}(\kappa_{\mathscr{D}}(\mathfrak{l}_1))\Big)
= \varphi\Big(
v_{\mathfrak{l}_1}(\kappa_{\mathscr{D}}(\mathfrak{l}_2)) \Big)
\end{equation} up to multiplication by a unit in $\mathcal{O}_{f,n}$. It is
(\ref{eq;apply}) instead of (\ref{eq:sec}) itself that is used in
the Euler system argument. We choose $\mathfrak{l}_1$ and
$\mathfrak{l}_2$ carefully to ensure (\ref{eq:cond}), so that we can
apply (\ref{eq;apply}).

For (\ref{eq:new-forms}), without ``Ihara Lemma'', the order of
$\gamma$, denoted by $n_0$ in our Proposition \ref{thm:second}, may
be not zero. Fortunately, we can bound $n_0$ by
$v_{\mathfrak{l}_2}(\kappa_{\mathscr{D}}(\mathfrak{l}_1))$.
Especially, for our good choice of $\mathfrak{l}_1$ and
$\mathfrak{l}_2$ we have $n_0<n$. Then we obtain from $\gamma$ an
$(N, n-n_0)$-admissible form denoted by $g''$ such that
\begin{equation} \label{eq:weak-new-form}
v_{\mathfrak{l}_2}(\kappa_{\mathscr{D}}(\mathfrak{l}_1))
=\omega^{n_0}\theta(g''). \end{equation} Thus we have a weaker
version of (\ref{eq:new-forms}). In the (inductive) Euler system
argument, the relation (\ref{eq:new-forms}) is used to show that
$\varphi \big(
v_{\mathfrak{l}_2}(\kappa_{\mathscr{D}}(\mathfrak{l}_1))\big)$
bounds a module $S_{\mathfrak{l}_1,\mathfrak{l}_2}^\vee$ (see
Section \ref{sec:BD-argument} for the meaning of this notation),
since this module is bounded by $\varphi(\theta(g''))$ by the
inductive assumption. Clearly, our weaker version
(\ref{eq:weak-new-form}) is sufficient for this purpose.

\vskip 10pt

A part of this paper was prepared when the author visited Professor
R. Sujatha and Department of Mathematics at the University of
British Columbia. The author thanks Professor R. Sujatha for her
kind hospitality.

The author thanks C.-H. Kim for pointing out mistakes in the
original version of this paper, and thanks Haining Wang for
discussions on his thesis \cite{Wang}.

\section*{Notations}

Fix a prime number $p\nmid \mathfrak{n}D_{K/F}$ and a prime ideal
$\mathfrak{p}$ of $\mathcal{O}_F$ above $p$.

Let $\widetilde{K}_m$ be the ring class field over $K$ of conductor
$\mathfrak{p}^m$ and put $G_m=\mathrm{Gal}(\widetilde{K}_m/K)$. Set
$\widetilde{K}_\infty=\cup_m \widetilde{K}_m$.

Let $K_\infty$ be the unique subfield of $\widetilde{K}_\infty$ such
that $\Gamma:=\mathrm{Gal}(K_\infty/K)\simeq
\mathbb{Z}_p^{[F_\mathfrak{p}:\mathbb{Q}_p]}$. Put
$K_m=\widetilde{K}_m\cap K_\infty$ and
$\Gamma_m=\mathrm{Gal}(K_m/K)$.

Let $\epsilon$ denote the $p$-adic cyclotomic character of
$G_F=\mathrm{Gal}(\overline{F}/F)$.

We will fix an isomorphism $\jmath: \mathbb{C}\cong \mathbb{C}_p$.

\section{Automorphic forms and Galois representations}

\subsection{Galois representation attached to $f$} \label{ss:Galois}

Throughout this paper we will fix a Hilbert cusp newform $f$ of
parallel weight $k\geq 2$ and trivial central character. Let
$\mathfrak{n}$ be the conductor of $f$, and we decompose
$\mathfrak{n}$ into $\mathfrak{n}=\mathfrak{n}^+\mathfrak{n}^-$
where $\mathfrak{n}^+$ is the product of primes split in $K$, and
$\mathfrak{n}^-$ is the product of primes inert or ramified in $K$.
We assume that $\mathfrak{n}$ is coprime to $p$.

We assume that $\mathfrak{n}^-$ satisfies the following two
conditions:

(sq-fr) $\mathfrak{n}^-$ is square-free, that is, $\mathfrak{n}^-$
is the product of different primes.

(card) The cardinal number of prime factors of $\mathfrak{n}^-$ has
the same parity as $[F:\mathbb{Q}]$.

By \cite{Wiles, Tay} (among other references) up to isomorphisms
there exists a unique $p$-adic Galois representation
$$\rho_f:G_F\rightarrow \mathrm{GL}_2(\mathbb{C}_p)$$ that satisfies the
following two properties.
\begin{quote} $\bullet$ $\rho_f$ is unramified outside of $p\mathfrak{n}$.
\\
$\bullet$ If $\mathfrak{l}$ is a prime of $\mathcal{O}_F$ not
dividing $p\mathfrak{n}$, then for the Frobenius element
$\mathrm{Frob}_\mathfrak{l}$ at $\mathfrak{l}$, the characteristic
polynomial of $\rho_f(\mathrm{Frob}_\mathfrak{l})$ is
$x^2-a_\mathfrak{l}(f) x + \mathbf{N}(\mathfrak{l})^{k-1}$.  Here,
$a_{\mathfrak{l}}(f)$ is the Hecke eigenvalue of $f$ at
$\mathfrak{l}$. \end{quote} Here we view $a_\mathfrak{l}(f)$ as an
element of $\mathbb{C}_p$ via $\jmath$. Let $E_f$ be the defining
field of $\rho_f$, which contains all $a_\mathfrak{l}(f)$.

A consequence of the latter property is
$$\det \rho_f= \epsilon^{k-1}.$$
The reader may consult
\cite{Wiles, Tay} for the construction of $\rho_f$ and more
properties of $\rho_f$.

Let $$\rho_f^*= \rho_f \otimes \epsilon^{\frac{2-k}{2}}$$ be the
self-dual twist of $\rho_f$, $V_f$ the underlying representation
space for $\rho_f^*$.  The representation $\rho_f^*$ has the
following properties.
\begin{quote}
$\bullet$ $\rho_f^*$ is unramified outside of $p\mathfrak{n}$. \\
$\bullet$
$\rho_f^*|_{G_{F_v}}=\wvec{\chi_v^{-1}\epsilon^{\frac{k}{2}}}{*}{0}{\chi_v
\epsilon^{\frac{2-k}{2}}}$ for any $v|p$. Here $\chi_v$ is the
unramified
character such that $\chi_v(\mathrm{Frob}_v)=a_v(f)$.\\
$\bullet$ $\rho_f^*|_{G_{F_\mathfrak{l}}}=\wvec{\pm
\epsilon}{*}{0}{\pm 1}$ for any $\mathfrak{l}$ dividing
$\mathfrak{n}$ exactly once.
\end{quote}

Fix a $G_F$-stable lattice $T_f$ of $V_f$, and put $A_f=V_f/T_f$.
For each positive integer $n$ we put
$\mathcal{O}_{f,n}=\mathcal{O}_f/\omega^n$, where $\omega$ is a
uniformizer of $\mathcal{O}_f$. We set
$$T_{f,n}=T_f/\omega^n=T_f\otimes_{\mathcal{O}_f}\mathcal{O}_{f,n}$$
and $$A_{f,n}=\mathrm{ker}(A_f\xrightarrow{\omega^n}A_f).$$ They are
all $G_F$-modules. We also use $\bar{\rho}_f$ to denote the residue
Galois representation $T_{f,1}$.

We state the conditions $(\mathrm{CR}^+)$,
$(\mathfrak{n}^+\text{-}\mathrm{DT})$ and $(\mathrm{PO})$ in Theorem
\ref{thm:main}. \vskip 5pt

\noindent {\bf Hypothesis} $(\mathrm{CR}^+)$. 1. $p>k+1$ and
$(\#(\mathcal{O}_{F}/\mathfrak{p})^\times)^{k-1}>5$.

2. The restriction of $\bar{\rho}_f$ to $G_{F(\sqrt{p^*})}$ is
irreducible, where $p^*=(-1)^{\frac{p-1}{2}}p$.

3. $\bar{\rho}_f$ is ramified at $\mathfrak{l}$ if
$\mathfrak{l}|\mathfrak{n}^-$ and $\mathrm{N}(\mathfrak{l})^2\equiv
1 \  (\mathrm{mod} \  p)$.

4. If $\mathfrak{n}_{\bar{\rho}}$ denotes the Artin conductor of
$\bar{\rho}_{f}$, then $\mathfrak{n}/\mathfrak{n}_{\bar{\rho}}$ is
coprime to $\mathfrak{n}_{\bar{\rho}}$. \vskip 5pt

\noindent {\bf Hypothesis} $(\mathrm{PO})$. $a_v^2(f)\
{\backslash\hskip -10pt \equiv } 1 \ (\mathrm{mod} \ p)$ for all
$v|p$ if $k=2$. \vskip 5pt

\noindent {\bf Hypothesis} $(\mathfrak{n}^+\text{-}\mathrm{DT})$. If
$\mathfrak{l}||\mathfrak{n}^+$ and $\mathrm{N}(\mathfrak{l})\equiv 1
\ (\mathrm{mod} \ p)$, then $\bar{\rho}_f$ is ramified at
$\mathfrak{l}$. \vskip 5pt

We also need an auxiliary condition
$(\mathfrak{n}^+\text{-}\mathrm{min})$. \vskip 5pt

\noindent {\bf Hypothesis} $(\mathfrak{n}^+\text{-}\mathrm{min})$.
If $\mathfrak{l}|\mathfrak{n}^+$, then $\bar{\rho}_f$ is ramified at
$\mathfrak{l}$.

\begin{rem} By $(\mathrm{CR}^+$-$2)$, $\bar{\rho}_f$ is itself
irreducible. So, the $G_F$-stable lattice $T_f$ of $V_f$ is unique
up to isomorphisms. Hence, up to isomorphisms $T_{f,n}$ and
$A_{f,n}$ are independent of the choice of $T_f$.
\end{rem}

\begin{lem}\label{lem:divide} Assume that $(\mathrm{CR}^+$-$4)$ holds.
If $\mathfrak{l}|\mathfrak{n}^+$ and $\bar{\rho}_f$ is ramified at
$\mathfrak{l}$, then $H^0(F_\mathfrak{l}^{\mathrm{nr}}, A_f)$ is
divisible.
\end{lem}
\begin{proof} By \cite{Tay} the Weil-Deligne representation attached
to $\rho_{f,\mathfrak{l}}$ has Frobenius-simplification the
Weil-Deligne representation attached to $\pi_{f,\mathfrak{l}}$ via
local Langlands correspondence.  Thus the Artin conductor of
$\rho_{f,\mathfrak{l}}$ is equal to the conductor of
$\pi_{f,\mathfrak{l}}$ \cite{Hen}. As $\epsilon$ is unramified at
$\mathfrak{l}$, the Artin conductor of $\rho^*_{f,\mathfrak{l}}$ is
equal to that of $\rho_{f,\mathfrak{l}}$.

When $\bar{\rho}_{f}$ is ramified at $\mathfrak{l}$,
$(\mathrm{CR}^+$-$4)$ ensures that the conductor of
$\pi_{f,\mathfrak{l}}$ is equal to the Artin conductor of
$\bar{\rho}_{f,\mathfrak{l}}$. Therefore, the Artin conductor of
$\rho^*_{f,\mathfrak{l}}$ is equal to that of
$\bar{\rho}_{f,\mathfrak{l}}$. Our assertion follows.
\end{proof}

\begin{defn} $($\cite[Definition 2.2.1]{Wang}$)$ A prime ideal $\mathfrak{l}$ of $\mathcal{O}_F$ is said to be
{\it $n$-admissible} for $f$ if the following conditions hold.

$\bullet$ $\mathfrak{l}\nmid p\mathfrak{n}$.

$\bullet$ $\mathfrak{l}$ is inert in $K$.

$\bullet$ $\mathrm{N}(\mathfrak{l})^2-1$ is not divided by $p$.

$\bullet$ $\omega^n$ divides $\mathrm{N}(\mathfrak{l})^{\frac{k}{2}}
    +\mathrm{N}(\mathfrak{l})^{\frac{k-2}{2}}-\epsilon_\mathfrak{l}a_\mathfrak{l}(f)$, where $\epsilon_\mathfrak{l}=\pm 1$.
\end{defn}

\subsection{$(N,n)$-admissible form} \label{ss:n-adm}

In this subsection we recall the definition of $n$-admissible forms
\cite{Wang}.

Let $B_\Delta$ be a quaternion algebra over $F$ with discriminant
$\Delta$. Suppose $\Delta$ is coprime to $p$. For each $v\nmid
\Delta$ we fix an isomorphism $(B_\Delta)_v \cong M_2(F_v)$.

Let $\mathfrak{n}^+$ be an ideal of $\mathcal{O}_F$ which is coprime
to $p\Delta$, and let $R_{\mathfrak{n}^+}\subset B_\Delta$ be an
Eichler order of level $\mathfrak{n}^+$. For a fixed positive
integer $n$ we put
$$\mathfrak{U}=\mathfrak{U}_{\mathfrak{n}^+, \mathfrak{p}^N} = \left\{  x\in \widehat{R}^\times_{\mathfrak{n}^+}: x_\mathfrak{p}
\equiv \wvec{a}{b}{0}{a} \hskip 5pt (\mathrm{mod}\ \
\mathfrak{p}^N), \ a, b\in \mathcal{O}_{F_\mathfrak{p}} \right\}. $$
Let $\mathbb{T}_{B_\Delta}(\mathfrak{n}^+, \mathfrak{p}^N)$ be the
(commutative) Hecke algebra generated by
$$ \{ T_v, S_v : v\nmid \mathfrak{p}\mathfrak{n}^+ \Delta \}\cup\{ U_v: v|\mathfrak{p}\mathfrak{n}^+\Delta \} \cup\{ \langle a\rangle : a\in \mathcal{O}^\times_{F,\mathfrak{p}} \}. $$
Here, as usual, for $v\nmid \mathfrak{p}\mathfrak{n}^+ \Delta$
$$T_v = [\ \mathfrak{U} \wvec{\pi_v}{0}{0}{1} \mathfrak{U}\ ] , \ \ \ S_v=[\ \mathfrak{U} \wvec{\pi_v}{0}{0}{\pi_v} \mathfrak{U}\ ]
$$ where $\pi_v$ is a uniformizing element of $F_v$; for
$v|\mathfrak{p}\mathfrak{n}^+$,
$$U_v = [ \ \mathfrak{U} \wvec{\pi_v}{0}{0}{1} \mathfrak{U} \ ] ;
$$ for $v|\Delta$, we choose an element $\pi'_v$ of $(B_\Delta)_v$
whose norm is a uniformizing element of $F_v$, and put
$$ U_v = [ \ \mathfrak{U} \pi'_v \mathfrak{U} \ ]; $$ for
$v=\mathfrak{p}$
$$ \langle a \rangle = [\ \mathfrak{U} \wvec{a}{0}{0}{1} \mathfrak{U} \ ] . $$

To define $n$-admissible form we need the notion of $p$-adic modular
form.

Let $\Phi$ be a finite extension of $\BQ_p$ that contains the image
of all embeddings $\sigma: F\hookrightarrow \overline{\BQ}_p$. Let
$\Omega$ be the maximal ideal of $\mathcal{O}_{\overline{\BQ}_p}$.
Then $\mathfrak{p}_\sigma:=\sigma^{-1}(\Omega)$ is a maximal ideal
of $\mathcal{O}_F$ lying above $p$. We extend $\sigma$ continuously
to $F_{\mathfrak{p}_\sigma}$. Let $A$ be an
$\mathcal{O}_\Phi$-algebra. Then we have a decomposition
$$ A\otimes_\BZ \mathcal{O}_F \cong \bigoplus_\sigma A , \hskip 5pt
a\otimes b \mapsto (a\sigma(b))_\sigma
$$ where $\sigma$ runs over all embeddings $F\hookrightarrow
\Phi$.

For each embedding $\sigma$ let
$$ L_{k,\sigma}(A)=A [X_\sigma,
Y_\sigma]_{k-2} $$ be the space of homogenous polynomials of degree
$k-2$ with two variables over $A$; we have an action of
$M_2(\mathcal{O}_{F_{\mathfrak{p}_\sigma}})$ on $L_{k}(A)$ by
$$ \widehat{\rho}_{k,\sigma}(g) P(X_\sigma, Y_\sigma)=  P((X_\sigma, Y_\sigma)g).
$$ We use $\rho_{k,\sigma}$ to denote the action $\det ^{\frac{2-k}{2}} \cdot
\widehat{\rho}_{k,\sigma}|_{\mathrm{GL}_2(\mathcal{O}_{F_{\mathfrak{p}_\sigma}})}$
of $\mathrm{GL}_2(\mathcal{O}_{F_{\mathfrak{p}_\sigma}})$. Put
$$L_k(A)=\bigotimes_\sigma L_{k,\sigma}(A) \cong
\bigotimes_{\mathfrak{q}|p}\bigotimes_{\sigma:\mathfrak{p}_\sigma=\mathfrak{q}}L_{k,\sigma}(A).$$
Then we consider $L_k(A)$ as a
$\mathrm{GL}_2(\mathcal{O}_{F_p})$-module by the action
$$\rho_k(u_p)=\bigotimes_{\mathfrak{q}|p}\bigotimes_{\sigma:\mathfrak{p}_\sigma=\mathfrak{q}}\rho_{k,\sigma}(\sigma(u_\mathfrak{q}))
.$$ Similarly, we consider $L_k(A)$ as a
$M_2(\mathcal{O}_{F_p})$-module by the action
$$\widehat{\rho}_k(u_p)=\bigotimes_{\mathfrak{q}|p}\bigotimes_{\sigma:\mathfrak{p}_\sigma=\mathfrak{q}}\widehat{\rho}_{k,\sigma}(\sigma(u_\mathfrak{q})).$$
Note that ${\rho}_k$ is self-dual; this means that there is
a ${\rho}_k$-invariant pairing $\langle\cdot,\cdot\rangle_k$
on $L_k(A)\times L_k(A)$.

Now, let $B_\Delta$ be definite. One defines the space
$S^{B_\Delta}_k(\mathfrak{U}, A)$ of {\it $p$-adic Hilbert modular
forms of level $\mathfrak{U}$ and weight $k$} by
$$ S^{B_\Delta}_k(\mathfrak{U}, A) = \left\{ f: B^\times\backslash \widehat{B}^\times\rightarrow L_k(A) \ |\ f(bu) = \rho_k(u_p)^{-1}f(b)  \ \ \ \forall\ u\in \mathfrak{U} \right\} . $$
It is equipped with a natural $\mathbb{T}_{B_\Delta}(\mathfrak{n}^+,
\mathfrak{p}^N)$-action: for any $ [ \mathfrak{U} x \mathfrak{U} ]
\in \mathbb{T}_{B_\Delta}(\mathfrak{n}^+, \mathfrak{p}^N)$, if
$x_{\mathfrak{p}}=\wvec{a}{0}{0}{1}$ with $a\in
\mathcal{O}_{F_\mathfrak{p}}^\times$ one defines
$$ [ \mathfrak{U} x \mathfrak{U} ] f ( b ) = \sum_{u\in \mathfrak{U}/\mathfrak{U}\cap x\mathfrak{U}x^{-1}} \rho_k(u_\mathfrak{p}x_{\mathfrak{p}})  f(bux);
$$ if $x_{\mathfrak{p}}=\wvec{\pi_\mathfrak{p}}{0}{0}{1}$ one defines
$$ [ \mathfrak{U} x \mathfrak{U} ] f ( b ) = \sum_{u\in \mathfrak{U}/\mathfrak{U}\cap x\mathfrak{U}x^{-1}}
\rho_k(u_\mathfrak{p})\widehat{\rho}_k(x_{\mathfrak{p}})
f(bux).
$$

When $k=2$, $S^{B_\Delta}_2(\mathfrak{U},A)$ can be naturally
identified with $A[B_\Delta^\times\backslash
\widehat{B}^\times_\Delta/\mathfrak{U}]$; it is compatible with
Hecke actions if we define the Hecke action on the divisor group of
the Shimura set $B_\Delta^\times\backslash
\widehat{B}^\times_\Delta/\mathfrak{U}$ via Picard functoriality.

Set $Y=\widehat{F}^\times$. Then there is an action of $Y$ on
$S^{B_\Delta}_k(\mathfrak{U}, A)$.

Let $f$ be the Hilbert modular form of level $\mathfrak{n}$ and
weight $k$ as in the introduction. In particular, $f$ is ordinary at
$p$. One can attach to $f$ a Hecke character $\lambda_{f,N}:
\mathbb{T}_{B_\Delta}(\mathfrak{n}^+, \mathfrak{p}^N)\otimes
\mathcal{O}_f \rightarrow \mathcal{O}_f$ as follows. As in Section
\ref{ss:Galois}, let $\{a_v(f)\}_v$ be the system of Hecke
eigenvalues attached to $f$. Set
\[\alpha_v(f) =\left\{ \begin{array}{ll} \text{the unit root of } x^2-a_v(f)x+\mathrm{N}(v)^{k-1} & \text{ if } v|p , \\
a_v(f)\mathrm{N}(v)^{\frac{2-k}{2}} & \text{ if } v\nmid p.
\end{array}\right.\]
Then we define $\lambda_{f,N}$ by
\begin{eqnarray*} \lambda_{f,N}(T_v) &=& a_v(f) , \\
\lambda_{f,N}(S_v) &=& 1 \hskip 5pt \text{ for } v\nmid p\mathfrak{n}, \\
\lambda_{f,N}(U_v) &=& \alpha_v(f) \hskip 5pt \text{ for } v| p\mathfrak{n}, \\
\lambda_{f,N}(\langle a \rangle) &=& a^{\frac{2-k}{2}} \hskip 5pt
\text{ for }a\in \mathcal{O}^\times_{F_\mathfrak{p}}.
\end{eqnarray*}

\begin{defn}\label{defn:adm}$($See \cite[Definition 5.1.1]{Wang}.$)$ Let $N$ and $n$ be two positive integers. By an {\it $(N,n)$-admissible form}  we mean a pair
$\mathscr{D}=(\Delta, g)$ such that

$\bullet$ $\Delta$ is a square-free product of prime ideals in
$\mathcal{O}_F$ inert in $K$; $\mathfrak{n}^-|\Delta$;
$\Delta/\mathfrak{n}^-$ is a product of $n$-admissible prime ideals,
and the cardinal number of prime factors of $\Delta/\mathfrak{n}^-$
is even;

$\bullet$ $g\in S_2^{B_\Delta}( \mathfrak{U}_{\mathfrak{n}^+,
\mathfrak{p}^N}, \mathcal{O}_{f,n})^Y$ such that $$g \:(\mathrm{mod}
\ \omega) \neq 0$$ and $$\lambda_g\equiv \lambda_{f,N} \
\mathrm{mod} \ \omega^n.$$

Let $\mathcal{I}_g$ be the kernel of $\lambda_g$.
\end{defn}

When $n=N$, $(N,n)$-admissible forms are just $N$-admissible forms
defined by \cite[Definition 5.1.1]{Wang}.

Let $\tau_N \in \widehat{B}^\times$ be the Atkin-Lehner element
given by
$$\tau_{N,v}=\wvec{0}{1}{\omega_v^{\mathrm{ord}_v(\mathfrak{p}^N\mathfrak{n}^+)}}{0}.$$
Then $\tau_N$ normalizes $\mathfrak{U}_{\mathfrak{n}^+,
\mathfrak{p}^N}$ and gives an involution, the Atkin-Lehner
involution, on $B_\Delta^\times\backslash
\widehat{B}^\times_\Delta/\mathfrak{U}_{\mathfrak{n}^+,
\mathfrak{p}^N}$. We define  a perfect pairing $$ \langle
\cdot,\cdot\rangle_{N}: \
S^{B_\Delta}_2(\mathfrak{U}_{\mathfrak{n}^+,
\mathfrak{p}^N},A)^{Y}\times
S_2^{B_\Delta}(\mathfrak{U}_{\mathfrak{n}^+,
\mathfrak{p}^N},A)^Y\rightarrow A $$ by
$$ \langle f,g \rangle_{N} =\sum_b f(b)g(b\tau_N)\sharp(B^\times \cap
b\mathfrak{U}_{\mathfrak{n}^+, \mathfrak{p}^N} b^{-1}/F^\times)^{-1}
$$ where $b$ runs over the Shimura set $B_\Delta^\times\backslash
\widehat{B}^\times_\Delta/\mathfrak{U}_{\mathfrak{n}^+,
\mathfrak{p}^N}$. The action of
$\mathbb{T}_{B_\Delta}(\mathfrak{n}^+, \mathfrak{p}^N)$ is
self-adjoint with respect to this pairing.

For each $g\in S^{B_\Delta}_2(\mathfrak{U}_{\mathfrak{n}^+,
\mathfrak{p}^N}, \mathcal{O}_{f,n})^Y$ define the map
$$ \psi_g: S^{B_\Delta}_2(\mathfrak{U}_{\mathfrak{n}^+,
\mathfrak{p}^N}, \mathcal{O}_f)^Y\rightarrow \mathcal{O}_{f,n},
\hskip 10pt h \mapsto \langle g,
h\rangle_{\mathfrak{U}_{\mathfrak{n}^+, \mathfrak{p}^N}}. $$ Via the
identity $$ S^{B_\Delta}_2(\mathfrak{U}_{\mathfrak{n}^+,
\mathfrak{p}^N}, \mathcal{O}_{f,n})^Y\cong
A[B_\Delta^\times\backslash
\widehat{B}^\times_\Delta/Y\mathfrak{U}_{\mathfrak{n}^+,
\mathfrak{p}^N}], $$ we have
\begin{equation}\label{eq:psi-g} \psi_g ( x \:\tau_N) = g (x).
\end{equation}

\begin{prop} \label{prop:revise} $($\cite[Proposition 5.1.2]{Wang}$)$ Assume  $\mathrm{CR}^+$ and
$(\mathfrak{n}^+\text{-}\mathrm{DT})$. If $n\leq N$, and if
$(\Delta,g)$ is an $(N,n)$-admissible form, then we have an
isomorphism
$$ \psi_g: S^{B_\Delta}_2(\mathfrak{U}_{\mathfrak{n}^+,
\mathfrak{p}^N}, \mathcal{O}_f)^Y / \mathcal{I}_g \xrightarrow{\sim}
\mathcal{O}_{f,n}. $$
\end{prop}
\begin{proof} When $n=N$, this is \cite[Proposition
5.1.2]{Wang}. For the general case $n\leq N$, we only need to
slightly adjust the proof of \cite[Proposition 5.1.2]{Wang}. Let
$P_k$ be the ideal $\{ \langle a\rangle -a^{\frac{k-1}{2}}: a\in
\mathcal{O}^\times_{F_\mathfrak{p}}\}$ which is clearly contained in
$\mathcal{I}_g$, and let $\mathfrak{m}$ be the maximal ideal
containing $\mathcal{I}_g$. In loc. cit. it is showed that
$$ S^{B_\Delta}_2(\mathfrak{U}_{\mathfrak{n}^+,
\mathfrak{p}^N}, \mathcal{O}_f)^Y_\mathfrak{m} / (P_k, \omega^N)
\simeq S^{B_\Delta}_k(\mathfrak{U}_{\mathfrak{n}^+},
\mathcal{O}_f)_\mathfrak{m}^Y/(\omega^N).
$$ Since $n\leq N$, it follows that $$ S^{B_\Delta}_2(\mathfrak{U}_{\mathfrak{n}^+,
\mathfrak{p}^N}, \mathcal{O}_f)^Y_\mathfrak{m} / (P_k, \omega^n)
\simeq S^{B_\Delta}_k(\mathfrak{U}_{\mathfrak{n}^+},
\mathcal{O}_f)_\mathfrak{m}^Y/(\omega^n).
$$ By \cite[Theorem 9.2.4]{Wang} $S^{B_\Delta}_k(\mathfrak{U}_{\mathfrak{n}^+}, \mathcal{O}_f)^Y_\mathfrak{m}$ is a cyclic
$\mathbb{T}_{B_\Delta}(\mathfrak{n}^+)$-module. Thus
$S^{B_\Delta}_2(\mathfrak{U}_{\mathfrak{n}^+, \mathfrak{p}^N},
\mathcal{O}_f)^Y/ \mathcal{I}_g$ is generated by some $h$ as a
$\mathbb{T}_{B_\Delta}(\mathfrak{n}^+,\mathfrak{p}^N)$-module. Since
$\psi_g$ is surjective, $\psi_g(h)\in \mathcal{O}^\times_{f,n}$. Now
our assertion follows from that
$\mathbb{T}_{B_\Delta}(\mathfrak{n}^+,\mathfrak{p}^N)$ is
self-adjoint with respect to $\langle\cdot,\cdot\rangle_{N}$.
\end{proof}

\subsection{Gross points and Theta elements} \label{ss:gross-theta}

We define Gross points and Theta elements following \cite{Wang}.

If $\Delta$ in Section \ref{ss:n-adm} is a product of inert primes
in $K$, then $K$ can be embedded into $B_\Delta$. We choose a basis
of $B=K\oplus K J$ over $K$ such that

$\bullet$ $J^2=\beta\in F^\times$ is totally negative, and
$Jt=\bar{t}J$ for $t\in K$;

$\bullet$ $\beta\in (\mathcal{O}_{F_v}^\times)^2$ for all
$v|\mathfrak{p}\mathfrak{n}^+$ and $\beta\in
\mathcal{O}_{F_v}^\times$ for all $v|D_{K/F}$.

To define the Gross point we need to choose a precise isomorphism
$$\prod_{v\nmid \Delta} i_v: \widehat{B}^{(\Delta)}\rightarrow
M_2(\widehat{F}^{(\Delta)}).$$ For this we fix a CM type $\Sigma$ of
$K$. Choose an element $\vartheta$ such that

$\bullet$ $\mathrm{Im}(\sigma(\vartheta))>0$ for all $\sigma\in
\Sigma$;

$\bullet$ $\{1, \vartheta_v\}$ is a basis of $\mathcal{O}_{K_v}$
over $\mathcal{O}_{F_v}$ for all
$v|D_{K/F}\mathfrak{p}\mathfrak{n}$;

$\bullet$ $\vartheta$ is local uniformizer at primes $v$ that are ramified in
$K$.

\noindent Then we require that for each
$v|\mathfrak{p}\mathfrak{n}^+$, $i_v$ is given by
$$ i_v(\vartheta) =\wvec{T(\vartheta)}{-N(\vartheta)}{1}{0}, \hskip 5pt i_v(J) =\sqrt{\beta} \wvec{-1}{T(\vartheta)}{0}{1}
 $$ where $T(\vartheta)=\vartheta+\bar{\vartheta}$ and $N(\vartheta)=\vartheta\bar{\vartheta}$; for $v\nmid \mathfrak{p}\mathfrak{n}^+\Delta$,
$i_v(\mathcal{O}_{K_v})\subset M_2(F_v)$.

Now we define the Gross points. For $v|\mathfrak{n}^+$ we put
$\zeta_v=(\vartheta-\bar{\vartheta})^{-1}\wvec{\vartheta}{\bar{\vartheta}}{1}{1}\in
\mathrm{GL}_2(K_w)=\mathrm{GL}_2(F_v)$ if $v=w\bar{w}$ in $K$. If
$m$ is a positive integer we put
$$ \zeta_\mathfrak{p}^{(m)} =
\left\{\begin{array}{ll}
\wvec{\vartheta}{-1}{1}{0}\wvec{\omega_\mathfrak{p}^m}{0}{0}{1}  \in
\mathrm{GL}_2(K_\mathfrak{P})=\mathrm{GL}_2(F_\mathfrak{p}) & \text{
if } \mathfrak{p}=\mathfrak{P}\bar{\mathfrak{P}}, \\
\wvec{0}{1}{-1}{0}\wvec{\omega_\mathfrak{p}^m}{0}{0}{1} & \text{ if
}\mathfrak{p} \text{ is inert}.
\end{array}\right. $$ Set $\zeta^{(m)}=\zeta_{\mathfrak{p}}^{(m)}\prod_{v|\mathfrak{n}^+}\zeta_v \in
\widehat{B}_\Delta^{\times}$.

Let $R_m$ be the order $\mathcal{O}_F+\mathfrak{p}^m\mathcal{O}_K$
of $K$. Then $(\zeta^{(m)})^{-1}\widehat{R}_m^\times
\zeta^{(m)}\subset \mathfrak{U}_{\mathfrak{n}^+,\mathfrak{p}^N}$.
Thus we have a map \begin{eqnarray*} x_m: K^\times\backslash
\widehat{K}^\times/ Y \widehat{R}_m^\times &\rightarrow&
B_\Delta^\times\backslash \widehat{B}_\Delta^\times/ Y
\mathfrak{U}_{\mathfrak{n}^+,\mathfrak{p}^N} \\
a & \mapsto & [a\zeta^{(m)}] . \end{eqnarray*}

If $\mathscr{D}=(\Delta, g)$ is an $(N,n)$-admissible form $(n\leq
N)$, for each $m\geq n$ we define
$$\Theta_m(g)=\frac{1}{\alpha_\mathfrak{p}^m}\sum_{[a]_m\in G_m}
g(x_m(a))[a]_m \in \mathcal{O}_{f,n}[G_m],
$$ where $a\mapsto [a]_m$ is the map induced by the normalized
geometrical reciprocity law. These elements $\Theta_m(g)$ are
compatible in the sense that
$\pi_{m+1,m}(\Theta_{m+1}(g))=\Theta_m(g)$. Here, $\pi_{m+1,m}$ is
the quotient map $\mathcal{O}_{f,n}[\Gamma_{m+1}]\rightarrow
\mathcal{O}_{f,n}[\Gamma_m] $.

Let $\pi_m:G_m\rightarrow \Gamma_m$ be the natural map, and put
$$ \theta_m(g)=\pi_m(\Theta_m(g))\in \mathcal{O}_{f,n}[\Gamma_m].$$
Then $\theta_m(g)$ are compatible and thus define an element
$\theta_\infty(g)$ of $ \mathcal{O}_{f,n}[[\Gamma]]$.

Now we restrict to the case $\Delta=\mathfrak{n}^-$, and put
$B=B_{\mathfrak{n}^-}$. Let $\widehat{R}_{\mathfrak{n}^+}$ be an
Eichler order in $B$ of level $\mathfrak{n}^+$.

By Jacquet-Langlands correspondence we find a
$\mathbb{C}_p$-automorphic representation $\pi'$ for the group
$G=\mathrm{Res}_{F/\mathbb{Q}}B^\times$ corresponding to $f$ (more
precisely $\jmath(f)$) and an eigenform $f_B\in
S_k^B(\widehat{R}^\times_{\mathfrak{n}^+}, \mathbb{C}_p)$ with the
property $T_vf_B=a_v(f)$ for $v\nmid \mathfrak{n}$ and
$U_vf_B=\alpha_v(f)f_B$ for $v|\mathfrak{n}$. Put
$$ \varphi_B(x)=\langle \rho_{k,\infty}(x_\infty)\mathbf{v}_0,
f_B(x^\infty) \rangle_k $$ where
$\mathbf{v}_0=X^{\frac{k-2}{2}}Y^{\frac{k-2}{2}}$. Then $\varphi_B$
is in the $\pi'$-part of the space of $\mathbb{C}_p$-automorphic
forms for $G$. We normalize $f_B$ such that $\varphi_B$ takes values
in $\mathcal{O}_f$ (enlarging $E_f$ if necessary) and is nonzero
modulo $\omega$.

Define the $\mathfrak{p}$-stabilization $\varphi^\dagger_B$ of
$\varphi_B$ as
$$ \varphi_B^\dagger = \varphi_B-\frac{1}{\alpha_\mathfrak{p}} \pi'(\wvec{1}{0}{0}{\omega_\mathfrak{p}})\varphi_B . $$
Then we define
$$\Theta_m(f)=\frac{1}{\alpha_\mathfrak{p}^m}\sum_{[a]_m\in G_m}
\varphi_B^\dagger(x_m(a))[a]_m \in \mathcal{O}_f[G_m].
$$  These elements $\Theta_m(f)$ are
compatible in the sense that
$\pi_{m+1,m}(\Theta_{m+1}(f))=\Theta_m(f)$ if $\pi_{m+1,m}$ denotes the quotient map $G_{m+1}\rightarrow G_m$. Then we define
$\theta_m(f)$ and $\theta(f)$ as above.

Finally we define the $p$-adic $L$-adic function $L_p(K_\infty,f)$
by $L_p(K_\infty, f)=\theta(f)^2$. Hung \cite{Hun} proved an
interpolation formula for $L_p(K_\infty, f)$. We do not state it
here, since we will not use it.

\begin{prop} $($\cite[Theorem 6.9]{Hun}$)$ \label{prop-Hung}
The analytic $\mu$-invariant of $L_p(K_\infty, f)$ is nonzero, i.e.
$L_p(K_\infty, f) \hskip 5pt / \hskip -10pt \equiv  0 \mod \omega$. In particular,
$L_p(K_\infty,f)\neq 0$.
\end{prop}

\begin{prop}\label{prop:exist-adm} $($\cite[Proposition 7.4.2]{Wang}$)$
If $\Delta=\mathfrak{n}^-$, there exists an $(N,N)$-admissible form
$\mathscr{D}_N=(\mathfrak{n}^-, f_N^\dagger)$ such that
$$\theta_m(\mathscr{D}_N)\equiv \theta_m(f) \hskip 10pt (\mathrm{mod}\ \omega^N)
$$ for each $m\geq N$. In particular $$\theta(\mathscr{D}_N)\equiv \theta(f) \hskip 10pt (\mathrm{mod}\ \omega^N)
$$
\end{prop}

\section{Selmer groups}

For the convenience of readers, we recall the definition of Selmer
groups. See \cite{BD05,CH15,Longo,Wang} for more details.

\subsection{Basic properties of Selmer groups}

Let $L$ be a finite extension of $F$. For each place $\mathfrak{l}$
of $F$ and each discrete $G_F$-module $M$, we put
$$ H^1(L_\mathfrak{l},M)= \bigoplus_{\lambda|\mathfrak{l}} H^1(L_\lambda, M), \hskip
10pt  H^1(I_{L_\mathfrak{l} }, M) = \bigoplus_{\lambda|\mathfrak{l}
} H^1(I_{L_\mathfrak{l}}, M), $$ where $\lambda$ runs through all
places of $L$ above $\mathfrak{l}$. Denote by
$$\res_\mathfrak{l}: H^1(L,M)\rightarrow H^1(L_\mathfrak{l}, M)$$ the restriction
map at $\mathfrak{l}$.

We define the {\it finite part} $H^1(L_\mathfrak{l},M)$ as
$$ H^1_\fin(L_\mathfrak{l},M) = \ker(H^1(L_\mathfrak{l},M)\rightarrow H^1(I_{L_\mathfrak{l}}, M))
$$ and the {\it singular quotient} as
$$ H^1_\sing(L_\mathfrak{l},M) = H^1(L_\mathfrak{l},M)/H^1_\fin(L_\mathfrak{l},M) .
$$ From the Hochschild-Serre spectral sequence  $$ E^{ij}_2=
H^i(G_{L_\lambda}/I_{L_\lambda}, H^j(I_{L_\lambda},
M))\Longrightarrow H^{i+j}(G_{L_\lambda},M) $$  we obtain the the
following exact sequence
\[ \xymatrix{  \bigoplus_{\lambda|\mathfrak{l}} H^1(G_{L_\lambda}/I_{L_\lambda}, M^{I_{L_\lambda}})
\ar[r] & H^1(L_\mathfrak{l},M) \ar[r]^{\hskip -30pt
\partial_\mathfrak{l}} &
 \bigoplus_{\lambda|\mathfrak{l}} H^1(I_{L_\lambda}, M)^{G_{L_\lambda}/I_{L_\lambda}} . } \]
Then $H^1_\fin(L_\mathfrak{l},M)$ coincides with the image of the
map $$ \bigoplus_{\lambda|\mathfrak{l}}
H^1(G_{L_\lambda}/I_{L_\lambda}, M^{I_{L_\lambda}}) \rightarrow
H^1(L_\mathfrak{l}, M), $$  and $H^1_\sing(L_\mathfrak{l},M)$ is
naturally isomorphic to the image of the residue map
$\partial_{\mathfrak{l}}.$ By abuse of notation, the composition map
$\partial_\mathfrak{l}\circ \res_\mathfrak{l}$ is also denoted by
$\partial_\mathfrak{l}$. If an element $s\in H^1(G_L, M)$ satisfies
$\partial_\mathfrak{l}(s)=0$, then $\res_{\mathfrak{l}}(s)$ is in
$H^1_\fin(L_\mathfrak{l}, M)$ and we will denote it as
$v_{\mathfrak{l}}(s)$.

If $\mathfrak{l}|\mathfrak{n}^-$, if $\mathfrak{l}$ is
$n$-admissible, or if $\mathfrak{l}|p$, then the restriction
$\rho^*_f|_{G_{F_\mathfrak{l}}}$ of $A_{f,n}$ to
$G_{F_\mathfrak{l}}$ sits in a $G_{F_\mathfrak{l}}$-equivariant
short exact sequence of free $\mathcal{O}_{f,n}$-modules
\[\xymatrix{ 0 \ar[r] & F^+_\mathfrak{l}A_{f,n} \ar[r] & A_{f,n} \ar[r] & F^-_\mathfrak{l}A_{f,n} \ar[r] & 0
, }\] where $G_{F_\mathfrak{l}}$ acts on $F^+_\mathfrak{l}A_{f,n}$
by $\pm \epsilon$ (resp. $\chi^{-1}\epsilon^{k/2}$) if
$\mathfrak{l}|\mathfrak{n}^{-}$ or $\mathfrak{l}$ is
$n$-admissible, (resp. $\mathfrak{l}|p$). Here, when $\mathfrak{l}|p$, $\chi$ is the
unramified character of $G_{F_{\mathfrak{l}}}$
such that $\chi(\mathrm{Frob})=\alpha_\mathfrak{l}$, where
$\alpha_\mathfrak{l}$ is the unit root of the Hecke polynomial
$x^2-a_\mathfrak{l}(f) x +N(\mathfrak{l})^{k-1}$. Then we define the
{\it ordinary part} of $H^1_\ord(L_\mathfrak{l}, A_{f,n})$ to be the
image of
$$ H^1(G_{L_\mathfrak{l}}, F^+_\mathfrak{l}A_{f,n})\rightarrow H^1(G_{L_\mathfrak{l}}, A_{f,n}).$$ We define $H^1_\ord(L_\mathfrak{l},T_{f,n})$ similarly.

Let $\Delta$ be a square free product of prime ideals in
$\mathcal{O}_F$ such that $\Delta/\mathfrak{n}^-$ is a product of
$n$-admissible prime ideals. Let $S$ be a finite (maybe empty) set
of places of $F$ that are coprime to $p\Delta \mathfrak{n}$.

\begin{defn} For $M=A_{f,n}$ or $T_{f,n}$ we define the Selmer group
$\Sel^S_{\Delta}(G_L, M)$ to be the the group of elements $s\in
H^1(G_L,M)$ such that

$\bullet$ $\res_\mathfrak{l}(s)\in H^1_\fin(L_\mathfrak{l},M)$ if
$\mathfrak{l}\nmid p\Delta$ and $\mathfrak{l}\notin S$;

$\bullet$ $\res_\mathfrak{l}(s)\in H^1_\ord(L_\mathfrak{l}, M)$ for
all $\mathfrak{l}| p\Delta$;

$\bullet$ $\res_\mathfrak{l}(s)$ is arbitrary if $\mathfrak{l}\in
S$.
\end{defn}

The group $\Gal(K_m/F)$ acts on $H^1(K_m, T_{f,n})$ and $H^1(K_m,
A_{f,n})$.

\begin{lem}\label{lem:preserve }
$\Gal(K_m/F)$ preserves $\Sel^S_{\Delta}(G_{K_m}, T_{f,n})$ and
$\Sel^S_{\Delta}(G_{K_m}, A_{f,n})$.
\end{lem}
\begin{proof}
If $\mathfrak{l}\nmid p\Delta$ and if $\mathfrak{l}\notin S$, then
for each place $\lambda$ of $K_m$ above $\mathfrak{l}$, the largest
unramified extension of $K_{m,\lambda}$ is Galois over
$F_\mathfrak{l}$. Thus $\mathrm{Gal}(K_m/F)$ acts on
$$
\bigoplus_{\lambda|\mathfrak{l}}H^1(G_{K_{m,\lambda}}/I_{K_{m,\lambda}},
T_{f,n}^{I_{K_{m,\lambda}}}) $$ and thus preserves $
H^1_\fin(K_{m,\mathfrak{l}}, T_{f,n})$.

If $\mathfrak{l}|p\Delta$, then $\Gal(K_m/F)$ preserves
$H^1_\ord(K_{m,\mathfrak{l}}, T_{f,n})$. This follows from the fact
that $G_{F_\mathfrak{l}}$ preserves the subspace $F^+_\mathfrak{l}
T_{f,n}$ of $T_{f,n}$ used to define the ordinary part. \end{proof}

\begin{prop}\label{prop:tool} $($\cite[Proposition 7.5]{Longo}, \cite[Theorem
7.1.2]{Wang}$)$ Assume the condition  $(\mathrm{CR}^+)$  holds. Let
$t\leq n$ be positive integers. Let $\kappa$ be a nonzero element in
$H^1(K,T_{f,t})$. Then there exist infinitely many $n$-admissible
primes $\mathfrak{l}$ such that $\partial_\mathfrak{l}(\kappa)=0$
and the map $$ v_\mathfrak{l}: \langle\kappa \rangle\rightarrow
H^1_\fin(K_{\mathfrak{l}}, T_{f,t})
$$ is injective, where $\langle\kappa \rangle$ denotes the
$\mathcal{O}_f$-submodule of $H^1(K,T_{f,t})$ generated by $\kappa$.
\end{prop}

We put \begin{eqnarray*} && H^1(K_\infty, A_{f,n}) =
\lim_{\overrightarrow{\;\;r\;\;}} H^1 (K_r, A_{f,n}), \hskip 10pt
\widehat{H}^1(K_\infty, T_{f,n}) = \lim_{\overleftarrow{\;\;
m\;\;}} H^1 (K_m, T_{f,n}) \\
&& H^1(K_{\infty,\mathfrak{l}}, A_{f,n}) =
\lim_{\overrightarrow{\;\;m\;\;}} H^1 (K_{m,\mathfrak{l}}, A_{f,n}),
\hskip 10pt \widehat{H}^1(K_{\infty,\mathfrak{l}}, T_{f,n}) =
\lim_{\overleftarrow{\;\; m\;\;}} H^1 (K_{m,\mathfrak{l}}, T_{f,n})
\end{eqnarray*} The finite parts and the singular quotients
$H^1_?(K_{\infty,\mathfrak{l}}, A_{f,n})$ and
$\widehat{H}^1_{?}(K_{\infty,\mathfrak{l}}, T_{f,n})$ for $?\in
\{\fin, \sing\}$ are defined similarly.

For each $\mathfrak{l}$ we have the local Tate pairing
$$ \langle\cdot,\cdot\rangle_{\mathfrak{l}}: \widehat{H}^1(K_{\infty,\mathfrak{l}}, T_{f,n})\times {H}^1(K_{\infty,\mathfrak{l}}, A_{f,n})\rightarrow E_f/\mathcal{O}_f. $$

\begin{prop}\label{prop-kill}\begin{enumerate}
\item\label{it:fine-zero} If $\mathfrak{l}$ splits in $K$, then $H^1_\fin(K_{\infty,\mathfrak{l}},
A_{f,n})=0$ and $\widehat{H}^1_\sing(K_{\infty,\mathfrak{l}},
T_{f,n})=0$.
\item\label{it:use-later} If $\mathfrak{l}$ is inert in $K$, then $\widehat{H}^1_{\sing}(K_{\infty,\mathfrak{l}}, T_{f,n})\cong H^1_\sing(K_\mathfrak{l},T_{f,n})\otimes \Lambda$.
\item\label{it:fine-kill} If $\mathfrak{l}\nmid p$, then
$H^1_\fin(K_{\infty,\mathfrak{l}}, A_{f,n})$ and
$\widehat{H}^1_\fin(K_{\infty,\mathfrak{l}}, T_{f,n})$ are
orthogonal to each other under the pairing $
\langle\cdot,\cdot\rangle_{\mathfrak{l}}$.
\item\label{it:free} If $\mathfrak{l}$ is $n$-admissible, then $\widehat{H}^1_\fin(K_{\infty,\mathfrak{l}},
T_{f,n})$, $\widehat{H}^1_\sing(K_{\infty,\mathfrak{l}}, T_{f,n})$
and  $\widehat{H}^1_\ord(K_{\infty,\mathfrak{l}}, T_{f,n})$ are free
of rank one over $\mathcal{O}_{f}[[\Gamma]]/(\omega^n)$.
\item\label{it:ord-kill}  Assume $(\mathrm{CR}^+)$ and $(\mathrm{PO})$ hold. If $\mathfrak{l}$ is
$n$-admissible or if $\mathfrak{l}|p\mathfrak{n}^-$, then
$H^1_\ord(K_{\infty,\mathfrak{l}}, A_{f,n})$ and
$\widehat{H}^1_\ord(K_{\infty,\mathfrak{l}}, T_{f,n})$ are
orthogonal to each other under the pairing $
\langle\cdot,\cdot\rangle_{\mathfrak{l}}$.
\end{enumerate}
\end{prop}
\begin{proof} This is \cite[Proposition 2.4.1,  Lemma 2.4.2, Proposition
2.4.4]{Wang}.
\end{proof}

We define
\begin{eqnarray*} && \mathrm{Sel}_\Delta^S(K_\infty, A_{f,n}) =
\lim_{\overrightarrow{\;\;m\;\;}} \mathrm{Sel}_\Delta^S (K_m,
A_{f,n}), \hskip 10pt \widehat{\mathrm{Sel}}_\Delta^S(K_\infty,
T_{f,n}) = \lim_{\overleftarrow{\;\; m \;\;}}\mathrm{Sel}_\Delta^S
(K_m, T_{f,n}) .
\end{eqnarray*}

If $S$ is empty, we drop $S$ from the above notations. When
$S=\emptyset$ and $\Delta=\mathfrak{n}^-$, we drop both $S$ and
$\Delta$ from the notations; the Selmer group in Theorem
\ref{thm:main} is in this case.

\subsection{Control theorems}

\begin{lem}\label{lem:control} Assume that $(\mathrm{CR}^+)$ holds. Let $L/K$ be a finite
extension contained in $K_\infty$.
\begin{enumerate}
\item\label{it:control-a} The the restriction maps  $$H^1(K, A_{f,n})\rightarrow H^1(L,
A_{f,n})^{\mathrm{Gal}(L/K)}$$ and $$ \mathrm{Sel}_\Delta^S(K,
A_{f,n})\rightarrow \mathrm{Sel}_\Delta^S(L,
A_{f,n})^{\mathrm{Gal}(L/K)}
$$ are isomorphisms.
\item\label{it:control-b} If $S$ contains all
prime  $\mathfrak{q}|\mathfrak{n}^+$ with
$\bar{\rho}_{f,\mathfrak{q}}$ unramified, then
\begin{equation}\label{eq:control} \mathrm{Sel}_\Delta^S(K,
A_{f,n})=\mathrm{Sel}_\Delta^S(K, A_{f})[\omega^n]. \end{equation}
In particular, if further $(\mathfrak{n}^+\text{-}\mathrm{min})$
holds, then for any set $S$ of primes, $($\ref{eq:control}$)$ holds.
\end{enumerate}
\end{lem}
\begin{proof} Assertion (\ref{it:control-a}) is \cite[Proposition
2.5.1(1)]{Wang}. We prove (\ref{it:control-b}).

Since $L/K$ is abelian, by $(\mathrm{CR}^+)$ we have
$A_{f,1}^{G_L}=0$. Then $A_{f,m}^{G_L}=0$ for every $m$, and
$A_f^{G_L}=0$. So from the exact sequence
$$\xymatrix{ 0 \ar[r] & A_{f,n} \ar[r] & A_f \ar[r]^{\omega^n} & A_f\ar[r] & 0
}$$ we obtain the isomorphism $H^1(G_L, A_{f,n})\cong H^1(G_L,
A_f)[\omega^n]$ and the injectivity of $\mathrm{Sel}_\Delta^S(L,
A_{f,n})\hookrightarrow \mathrm{Sel}_\Delta^S(L, A_{f})[\omega^n]$.
To prove the surjectivity of $\mathrm{Sel}_\Delta^S(L,
A_{f,n})\hookrightarrow \mathrm{Sel}_\Delta^S(L, A_{f})[\omega^n]$,
it suffices to prove

$(\text{i})$ $H^1(L^{\mathrm{ur}}_\mathfrak{l}, A_{f,n})\rightarrow
H^1(L^{\mathrm{ur}}_\mathfrak{l}, A_{f})$ is injective for
$\mathfrak{l}\nmid p\Delta$ and $\mathfrak{l}\notin S$.

$(\text{ii})$ $H^1(L_\mathfrak{l}, A_{f,n}/
F^+_\mathfrak{l}A_{f,n})\rightarrow H^1(L_\mathfrak{l},
A_{f}/F^+_\mathfrak{l}A_{f})$ is injective for $\mathfrak{l}|
p\Delta$.

For (\text{i}) if $\mathfrak{l}\nmid \mathfrak{n}^+$, the action of
$I_{L,\mathfrak{l}}$ is trivial and the claim follows immediately.
If $\mathfrak{l}| \mathfrak{n}^+$, then by Lemma \ref{lem:divide},
$H^0(F_\mathfrak{l}^{\mathrm{nr}}, A_f)$ is  divisible. The claim
again follows.

For (\text{ii}) if $\mathfrak{l}|\Delta$, the actions of
$G_{L_\mathfrak{l}}$ on $A_{f,n}/ F^+_\mathfrak{l}A_{f,n}$ and
$A_{f}/ F^+_\mathfrak{l}A_{f}$ are trivial and the claim is clear.
If $\mathfrak{l}|p$, then $G_{L_\mathfrak{l}}$ acts on $A_{f,m}/
F^+_\mathfrak{l}A_{f,m}$ by
$\chi_\mathfrak{l}\epsilon^{1-\frac{k}{2}}$, where
$\chi_\mathfrak{l}$ is an unramified character. Thus
$H^0(L_\mathfrak{l}, A_{f,m}/ F^+_\mathfrak{l}A_{f,m})=0$ for each
$m$. Then $H^0(L_\mathfrak{l}, A_{f}/ F^+_\mathfrak{l}A_{f})=0$. The
claim follows.
\end{proof}

\begin{thm}\label{prop-free} $($\cite[Proposition
7.2.3]{Wang}$)$ Assume the conditions $(\mathrm{CR}^+)$,
$(\mathrm{PO})$ and $(\mathfrak{n}^+\text{-}\mathrm{min})$ hold. For
each positive integer $n$ there exists a finite set $S$ of
$n$-admissible prime ideals such that
$\widehat{\mathrm{Sel}}^S_\Delta(K_\infty, T_{f,n})$ is free over
$\Lambda/(\omega^{n})$.
\end{thm}

\begin{thm}\label{prop-cond-free} If
$\mathrm{Sel}_{\mathfrak{n}^-}^{\mathfrak{n}^+}(K_\infty, A_f)$ is
$\Lambda$-cotorsion and the algebraic $\mu$-invariant of
$\mathrm{Sel}_{\mathfrak{n}^-}^{\mathfrak{n}^+}(K_\infty, A_f)^\vee$
vanishes, then for any set $S$ of $n$-admissible primes that does
not divide $p\mathfrak{n}\Delta$,
$\widehat{\mathrm{Sel}}_\Delta^{S}(K_\infty, T_{f,n})$ is free over
$\Lambda/\omega^n$.
\end{thm}
\begin{proof} This was essentially proved by Wang \cite[Chapter 10]{Wang} following Kim, Pollack
and Weston's idea \cite{Poll-West}. However, the assertion in the
above form is not clearly stated in loc. cit., so we give a sketch
of the proof.

Let
$$ \Phi_n: \{\text{cofinitely generated }\Lambda\text{-modules}\}\rightarrow \{ \text{finitely generated }\Lambda/\omega^n\text{-modules} \}
$$ be the functor defined by $\Phi(M)=\lim\limits_{\overleftarrow{\;\; m \;}} M[\omega^n]^{\Gamma_m}
$. It follows from Lemma \ref{lem:control} (\ref{it:control-a}) that
$$ \Phi_n (\mathrm{Sel}^{\mathfrak{n}^+S}_{\mathfrak{n}^-}(K_\infty, A_f))\cong
\lim\limits_{\overleftarrow{\;\; m \;}}\mathrm{Sel}^{\mathfrak{n}^+S}_{\mathfrak{n}^-}(K_\infty, A_f)[\omega^n]^{\Gamma_m}
= \lim\limits_{\overleftarrow{\;\; m
\;}}\mathrm{Sel}^{\mathfrak{n}^+S}_{\mathfrak{n}^-}(K_m, A_{f,n})
\cong
\widehat{\mathrm{Sel}}^{\mathfrak{n}^+S}_{\mathfrak{n}^-}(K_\infty,
T_{f,n}).
$$ The functor $\Phi_n$ satisfies the following properties.

$\bullet$ If $A$ and $B$ are pseudo-isomorphic cofinitely generated
$\Lambda$-modules, then $\Phi_n(A)=\Phi_n(B)$.

$\bullet$ If $Y$ is a finitely cotorsion $\Lambda$-module with
vanishing (algebraic) $\mu$-invariant, then $\Phi_n(Y)=0$.

$\bullet$ If $Y=\Lambda/\omega^t$ with $t\geq n$, then
$\Phi_n(Y^\vee)=\Lambda/\omega^n$.

Wang \cite[Lemma 10.1.2]{Wang} showed that for any set $S$ away from
$p\mathfrak{n}\Delta$,
$\mathrm{Sel}^{S\mathfrak{n}^+}_{\mathfrak{n}^-}(K_\infty, A_f)$
sits in the exact sequence
$$\xymatrix{ 0 \ar[r] & \mathrm{Sel}^{\mathfrak{n}^+}_{\mathfrak{n}^-}(K_\infty, A_f) \ar[r] &
\mathrm{Sel}^{\mathfrak{n}^+S}_{\mathfrak{n}^-}(K_\infty, A_f)\ar[r]
& \prod_{v\in S}\mathcal{H}_v\ar[r] & 0 }$$ where
$\mathcal{H}_v=\lim\limits_{\overrightarrow{\;\;m\;}}\prod_{w|v}H^1(K_{m,w},
A_f)$. When $v\in S$ is $n$-admissible, $\mathcal{H}_v\cong
(\Lambda/\omega^{t_v})^\vee$ for some $t_v\geq n$ \cite[Lemma
10.1.3]{Wang}. Thus
$\mathrm{Sel}^{\mathfrak{n}^+S}_{\mathfrak{n}^-}(K_\infty,
A_f)^\vee$ is pseudo-isomorphic to $$(\bigoplus\limits_{v\in
S}\Lambda/\omega^{t_v}) \times Y,
$$ where $Y$ is a torsion $\Lambda$-module with $\mu(Y)=0$. Hence, by the above properties of $\Phi_n$, $\widehat{\mathrm{Sel}}^{S\mathfrak{n}^+}_{\mathfrak{n}^-}(K_\infty,
T_{f,n})$ is free of rank $\sharp S$ over $\Lambda/\omega^n$.

For $\mathfrak{l}\in \Delta$ we have the following exact sequence
$$ \xymatrix{ 0\ar[r] & \widehat{\mathrm{Sel}}^{\mathfrak{n}^+S}_{\mathfrak{l}\mathfrak{n}^-}(K_\infty,
T_{f,N}) \ar[r] &
\widehat{\mathrm{Sel}}^{\mathfrak{l}\mathfrak{n}^+S}_{\mathfrak{n}^-}(K_\infty,
T_{f,n}) \ar[r] &
\widehat{H}^1_\fin(K_{\infty,\mathfrak{l}},T_{f,n})\ar[r] & 0 . }
$$ So, the freeness of $\widehat{\mathrm{Sel}}^{\mathfrak{l}\mathfrak{n}^+S}_{\mathfrak{n}^-}(K_\infty,
T_{f,n})$ and $\widehat{H}^1_\fin(K_{\infty,\mathfrak{l}}, T_{f,n})$
implies the freeness of
$\widehat{\mathrm{Sel}}^{\mathfrak{n}^+S}_{\mathfrak{l}\mathfrak{n}^-}(K_\infty,
T_{f,n})$. Repeating this several times we obtain the freeness of
$\widehat{\mathrm{Sel}}^{\mathfrak{n}^+S}_{\Delta}(K_\infty,
T_{f,n})$. By Proposition \ref{prop-kill} (\ref{it:fine-zero}), we
have $\widehat{\mathrm{Sel}}^{S}_{\Delta}(K_\infty,
T_{f,n})=\widehat{\mathrm{Sel}}^{\mathfrak{n}^+S}_{\Delta}(K_\infty,
T_{f,n})$.
\end{proof}

\section{Euler system of Heegner points}

Fix $N\geq n \geq 1$. Let $\mathscr{D}=(\Delta, g)$,
$\mathfrak{n}^-|\Delta$, be an $(N,n)$-admissible form.

\subsection{Shimura curves}

In this subsection we collect needed results on Shimura curves
\cite{Longo, Wang}.

Let $\mathfrak{l} \nmid \Delta$ be an $n$-admissible prime ideal of
$f$ with
$\epsilon_\mathfrak{l}\alpha_\mathfrak{l}=\mathrm{N}(\mathfrak{l})+1
\ (\mathrm{mod} \  \omega^n)$. One defines a character of Hecke
algebra
$$\lambda_g^{[\mathfrak{l}]}:
\mathbb{T}_{B_\Delta}(\mathfrak{l}\mathfrak{n}^+,
\mathfrak{p}^N)\otimes \mathcal{O}_{f,n}\rightarrow
\mathcal{O}_{f,n}$$ by
$\lambda_g^{[\mathfrak{l}]}(U_{\mathfrak{l}})=\epsilon_\mathfrak{l}$
, and let $\mathcal{I}_g^{[\mathfrak{l}]}$ be the kernel of
$\lambda_g^{[\mathfrak{l}]}$.

Let $B'=B'_{\Delta \mathfrak{l}}$ be a quaternion algebra with
discriminant $\Delta\mathfrak{l}$ that splits at exactly one real
place. Then we have an isomorphism $\phi:
\widehat{B}_\Delta^{(\mathfrak{l})} \cong
\widehat{B}'^{(\mathfrak{l})}$. Let $\mathcal{O}_{B'_\mathfrak{l}}$
be the maximal order of $B'_\mathfrak{l}$. Put
$$ \mathfrak{U}'=\mathfrak{U}'_{\mathfrak{n}^+,\mathfrak{p}^N}=\phi((\mathfrak{U}_{\mathfrak{n}^+,\mathfrak{p}^N})^{(\mathfrak{l})})\mathcal{O}_{B'_\mathfrak{l}}^\times. $$
With $\mathfrak{U}'$ instead of
$\mathfrak{U}=\mathfrak{U}_{\mathfrak{n}^+,\mathfrak{p}^N}$ we have
a Hecke algebra $\mathbb{T}_{B_{\Delta\mathfrak{l}}}(\mathfrak{n}^+,
\mathfrak{p}^N)$.

% Let $t':K\hookrightarrow B'$ be the embedding that induces the composition $\widehat{K}^{(\mathfrak{l})}\xrightarrow{t}
% B^{(\mathfrak{l})}\xrightarrow{\phi} B'^{(\mathfrak{l})}$.

Associated to $(B',Y\mathfrak{U}')$ there is a Shimura curve
$M_N^{[\mathfrak{l}]}$ with complex points
$$ M_N^{[\mathfrak{l}]}(\mathbb{C}) = B'^\times \backslash (\mathbf{P}^1(\mathbb{C})-\mathbf{P}^1(\mathbb{R}))\times \widehat{B}'^\times/ Y \mathfrak{U}'; $$
$M_N^{[\mathfrak{l}]}$ is smooth and projective over $F$. We write
$[z, b']_N$ for a point in $M_N^{[\mathfrak{l}]}$ corresponding to
$z\in \mathbf{P}^1(\mathbb{C})-\mathbf{P}^1(\mathbb{R})$ and $b'\in
\widehat{B}'^\times$.

Let $J^{[\mathfrak{l}]}_N$ be the Jacobian of
$M^{[\mathfrak{l}]}_N$, and let $\Phi^{[\mathfrak{l}]}$ be the
component group of the Neron model of $J_N^{[\mathfrak{l}]}$ over
$F_{\mathfrak{l}^2}$, the unramified extension of $F_{\mathfrak{l}}$
of degree $2$. Let $r_\mathrm{l}:J^{[\mathfrak{l}]}_N\rightarrow
\Phi^{[\mathfrak{l}]}$ be the reduction map.

There is a natural action of
$\mathbb{T}_{B_{\Delta\mathfrak{l}}}(\mathfrak{n}^+,
\mathfrak{p}^N)$ on $J^{[\mathfrak{l}]}_N$ via Picard functoriality.
Note that
$$\mathbb{T}_{B_\Delta}^{(\mathfrak{l})}(\mathfrak{l}\mathfrak{n}^+,
\mathfrak{p}^N)\simeq \mathbb{T}_{B_{\Delta
\mathfrak{l}}}^{(\mathfrak{l})}(\mathfrak{n}^+, \mathfrak{p}^N).$$
We extend it to a homomorphism
$$\varphi_*: \mathbb{T}_{B_\Delta}(\mathfrak{l}\mathfrak{n}^+,
\mathfrak{p}^N)\rightarrow \mathbb{T}_{B_{\Delta
\mathfrak{l}}}(\mathfrak{n}^+, \mathfrak{p}^N)$$ which sends
$U_\mathfrak{l}=[\mathfrak{U}\wvec{\pi_\mathfrak{l}}{0}{0}{1}\mathfrak{U}]$
to $U_\mathfrak{l}=[\mathfrak{U}' \pi'_\mathfrak{l} \mathfrak{U}']$.
Via $\varphi_*$ we obtain an action of
$\mathbb{T}_{B_{\Delta}}(\mathfrak{l}\mathfrak{n}^+,
\mathfrak{p}^N)$ on $J^{[\mathfrak{l}]}_N$. It induces an action of
$\mathbb{T}_{B_{\Delta}}(\mathfrak{l}\mathfrak{n}^+,
\mathfrak{p}^N)$ on $\Phi^{[\mathfrak{l}]}$.

We need the relation between $\Phi^{[\mathfrak{l}]}$ and the graph
$\mathcal{G}$ of the special fiber of $M_N^{[\mathfrak{l}]}$
\cite{BLR}.

The set of vertices  which correspond to irreducible components of
$M_N^{[\mathfrak{l}]}$ is identified with
$$ \mathcal{V}(\mathcal{G}) = B^\times \backslash \widehat{B}^\times/ Y \mathfrak{U}_{\mathfrak{n}^+,\mathfrak{p}^N} \times \mathbb{Z}/2\mathbb{Z}. $$
The set of oriented edges which correspond to the set of singular
points on the special fiber is identified with
$$\vec{\mathcal{E}}(\mathcal{G}) =  B^\times \backslash \widehat{B}^\times/ Y \mathfrak{U}_{\:\mathfrak{l}\mathfrak{n}^+,\mathfrak{p}^N} \times \mathbb{Z}/2\mathbb{Z} . $$

We choose an orientation of $\vec{\mathcal{E}}(\mathcal{G})$  such
that the source and target maps $s, t:
\mathcal{E}(\mathcal{G})\rightarrow \mathcal{V}(\mathcal{G})$ are
given by
\begin{eqnarray*}
s: \mathcal{E}(\mathcal{G}) = B^\times \backslash
\widehat{B}^\times/ Y
\mathfrak{U}_{\:\mathfrak{l}\mathfrak{n}^+,\mathfrak{p}^N}
&\rightarrow& \mathcal{V}(\mathcal{G}) = B^\times \backslash
\widehat{B}^\times/ Y \mathfrak{U}_{\mathfrak{n}^+,\mathfrak{p}^N}
\times \mathbb{Z}/2\mathbb{Z} \\
B^\times b \: Y
\mathfrak{U}_{\:\mathfrak{l}\mathfrak{n}^+,\mathfrak{p}^N} &\mapsto
& (B^\times b \: Y \mathfrak{U}_{\mathfrak{n}^+,\mathfrak{p}^N},0)
\end{eqnarray*} and
\begin{eqnarray*}
t: \mathcal{E}(\mathcal{G}) = B^\times \backslash
\widehat{B}^\times/ Y
\mathfrak{U}_{\:\mathfrak{l}\mathfrak{n}^+,\mathfrak{p}^N}
&\rightarrow& \mathcal{V}(\mathcal{G}) = B^\times \backslash
\widehat{B}^\times/ Y \mathfrak{U}_{\mathfrak{n}^+,\mathfrak{p}^N}
\times \mathbb{Z}/2\mathbb{Z} \\
B^\times b \: Y
\mathfrak{U}_{\:\mathfrak{l}\mathfrak{n}^+,\mathfrak{p}^N} &\mapsto
& (B^\times b \: Y \mathfrak{U}_{\mathfrak{n}^+,\mathfrak{p}^N},1).
\end{eqnarray*}

Let
$$d_*=t_*-s_*: \mathbb{Z}[\mathcal{E}(\mathcal{G})]\rightarrow
\mathbb{Z}[\mathcal{V}(\mathcal{G})]$$ be the boundary map, and
$$d^*:t^*-s^*:\mathbb{Z}[\mathcal{V}(\mathcal{G})]\rightarrow
\mathbb{Z}[\mathcal{E}(\mathcal{G})]$$ its dual. Put
$\mathbb{Z}[\mathcal{V}(\mathcal{G})]_0=\mathrm{im}(d_*)$. By
\cite[Section 9.6, Theorem 1]{BLR} there exists a natural
identification
$$ \Phi^{[\mathfrak{l}]} \simeq \mathbb{Z}[\mathcal{V}(\mathcal{G})]_0/ d_*d^*
.$$

One can identify $\mathbb{Z}[\mathcal{V}(\mathcal{G})]$ with
$(S^{B_\Delta}_2(\mathfrak{U}, \mathcal{O}_f)^Y)^{\oplus 2}$, and
identify $\mathbb{Z}[\mathcal{V}(\mathcal{G})]_0$ with a submodule
$(S^{B_\Delta}_2(\mathfrak{U}, \mathcal{O}_f)^Y)^{\oplus 2}_0$ of
$(S^{B_\Delta}_2(\mathfrak{U}, \mathcal{O}_f)^Y)^{\oplus 2}$. Define
an action of
$\mathbb{T}_{B_\Delta}(\mathfrak{l}\mathfrak{n}^+,\mathfrak{p}^N)$
on $(S^{B_\Delta}_2(\mathfrak{U}, \mathbb{Z})^Y)^{\oplus 2}$ by
$$t(x,y)=(t(x), t(y)) \hskip 30pt t\in \mathbb{T}_{B_\Delta}^{(\mathfrak{l})}(\mathfrak{l}\mathfrak{n}^+,\mathfrak{p}^N)$$
and
$$ \tilde{U}_\mathfrak{l}(x,y)= (-\mathrm{N}(\mathfrak{l})y, x+T_\mathfrak{l}(y))
.$$ In case of confusing with the diagonal action we use the
notation $\tilde{U}_\mathfrak{l}$ instead of ${U}_\mathfrak{l}$.

\begin{prop}\label{prop:Wang} $($\cite[Proposition 4.4.1]{Wang}$)$ We have the following
$\mathbb{T}_{B_\Delta}(\mathfrak{l}\mathfrak{n}^+,\mathfrak{p}^N)$-module
isomorphism
$$ \Phi^{[\mathfrak{l}]} \simeq (S^{B_\Delta}_2(\mathfrak{U}, \mathbb{Z})^Y)^{\oplus 2}_0 / (\tilde{U}_\mathfrak{l}^2-1). $$
\end{prop}

Write
$$ \Phi^{[\mathfrak{l}]}_{\mathcal{O}_f}=\Phi^{[\mathfrak{l}]}\otimes_{\mathbb{Z}}{\mathcal{O}_f}. $$

\begin{cor}\label{cor:expo} We have an isomorphism
$$ \Phi^{[\mathfrak{l}]}_{\mathcal{O}_f} /
 \mathcal{I}_g^{[\mathfrak{l}]} \simeq  S^{B_\Delta}_2(\mathfrak{U}, \mathcal{O}_f)^Y /\mathcal{I}_g \xrightarrow{\psi_g} \mathcal{O}_{f,n} .$$
\end{cor}

When $n=N$, this is \cite[Theorem 5.1.3]{Wang}.

\begin{proof} Let $\mathfrak{m}^{[\mathfrak{l}]}$ be the maximal
ideal of
$\mathbb{T}_{B_\Delta}(\mathfrak{l}\mathfrak{n}^+,\mathfrak{p}^N)$
containing $\mathcal{I}^{[\mathfrak{l}]}_g$.

Note that $(S^{B_\Delta}_2(\mathfrak{U}, \mathbb{Z})^Y)^{\oplus
2}/(S^{B_\Delta}_2(\mathfrak{U}, \mathbb{Z})^Y)^{\oplus 2}_0$ is
Eisenstein, while $\mathfrak{m}^{[\mathfrak{l}]}$ is not Eisenstein.
Thus $$ (S^{B_\Delta}_2(\mathfrak{U}, \mathbb{Z})^Y)^{\oplus 2}_{0\
\ \mathfrak{m}^{[\mathfrak{l}]}} =(S^{B_\Delta}_2(\mathfrak{U},
\mathbb{Z})^Y)^{\oplus 2}_{ \ \ \mathfrak{m}^{[\mathfrak{l}]}}.
$$ By Proposition \ref{prop:Wang}  we obtain
$$ ( \Phi^{[\mathfrak{l}]}_{\mathcal{O}_f})_{\mathfrak{m}^{[\mathfrak{l}]}} \simeq (S^{B_\Delta}_2(\mathfrak{U},
\mathbb{Z})^Y)^{\oplus 2}_{ \mathfrak{m}^{[\mathfrak{l}]}} /
(\tilde{U}_\mathfrak{l}^2-1) \simeq (S^{B_\Delta}_2(\mathfrak{U},
\mathbb{Z})^Y)^{\oplus 2}_{ \mathfrak{m}^{[\mathfrak{l}]}} /
(\tilde{U}_\mathfrak{l}-\epsilon_\mathfrak{l}).
$$ Hence,
\begin{eqnarray*}  \Phi^{[\mathfrak{l}]}_{\mathcal{O}_f} /
 \mathcal{I}_g^{[\mathfrak{l}]} & \simeq & \Big( (S^{B_\Delta}_2(\mathfrak{U},
\mathbb{Z})^Y)^{\oplus 2} /
(\tilde{U}_\mathfrak{l}-\epsilon_\mathfrak{l})\Big) \otimes
\mathbb{T}_{B_\Delta}(\mathfrak{l}\mathfrak{n}^+,\mathfrak{p}^N)/
\mathcal{I}^{[\mathfrak{l}]}_g  \\ & \simeq & \Big(
S^{B_\Delta}_2(\mathfrak{U}, \mathbb{Z})^Y  /
(\epsilon_\mathfrak{l}T_\mathfrak{l}-\mathbf{N}(\mathfrak{l})-1)\Big)
\otimes \mathbb{T}_{B_\Delta}(\mathfrak{n}^+,\mathfrak{p}^N)/
\mathcal{I}_g \\
&\simeq & S^{B_\Delta}_2(\mathfrak{U}, \mathbb{Z})^Y  /
\mathcal{I}_g .
\end{eqnarray*} By Proposition \ref{prop:revise}, $\psi_g$ is an isomorphism.
\end{proof}

Let $T_p(J_N^{[\mathfrak{l}]})$ be the $p$-adic Tate module of
$J_N^{[\mathfrak{l}]}$. Then $T_p(J_N^{[\mathfrak{l}]})$ is a
$\mathbb{T}_{B_\Delta}(\mathfrak{l}\mathfrak{n}^+,\mathfrak{p}^N)$-module.

\begin{prop} We have an isomorphism of $G_F$-modules
$$ T_p(J_N^{[\mathfrak{l}]})_{\mathcal{O}_f}/\mathcal{I}_g^{[\mathfrak{l}]} \simeq T_{f,n} . $$
\end{prop}

When $n=N$, this is \cite[Theorem 5.1.4]{Wang}.
\begin{proof} Let $\mathfrak{m}^{[\mathfrak{l}]}$ be the maximal
ideal of
$\mathbb{T}_{B_\Delta}(\mathfrak{l}\mathfrak{n}^+,\mathfrak{p}^N)$
containing $\mathcal{I}^{[\mathfrak{l}]}_g$. In the proof of
\cite[Theorem 5.1.4]{Wang} it is showed that
$T_p(J_N^{[\mathfrak{l}]})_{\mathcal{O}_f}/\mathfrak{m}^{[\mathfrak{l}]}
\simeq T_{f,1}$. By irreducibility of $T_{f,1}$, to finish the proof
one only needs to show that the exponent of
$T_p(J_N^{[\mathfrak{l}]})_{\mathcal{O}_f}/\mathcal{I}_g^{[\mathfrak{l}]}$
is $\omega^n$. On one hand, its exponent is at most $\omega^n$. On
the other hand, when $n'$ is sufficiently large,
$J_N^{[\mathfrak{l}]}[p^{n'}]/\mathcal{I}_g^{[\mathfrak{l}]}$ maps
onto $\Phi^{[\mathfrak{l}]}/\mathcal{I}_g^{[\mathfrak{l}]}$,
together with Corollary \ref{cor:expo} which implies that the
exponent is at least $\omega^n$.
\end{proof}

Let $$ \mathrm{Kum}:
J^{[\mathfrak{l}]}_N(K_m)_{\mathcal{O}_f}\rightarrow H^1(K_m,
T_p(J^{[\mathfrak{l}]}_N)_{\mathcal{O}_f}) $$ be the Kummer map.

\begin{prop}\label{prop:comm} $($\cite[Theorem
5.2.2]{Wang}$)$ We have the following commutative diagram
\[ \xymatrix{ J_N^{[\mathfrak{l}]}(K_m)_{\mathcal{O}_f}/ \mathcal{I}_g^{[\mathfrak{l}]} \ar[r]^{\mathrm{Kum}}\ar[d]^{r_\mathfrak{l}} & H^1(K_m, T_{f,n})\ar[d]^{\partial_\mathfrak{l}} \\
\Phi^{[\mathfrak{l}]}_{\mathcal{O}_f} \ar[r]^{\hskip -15pt \psi_g
}_{\hskip -15pt \simeq} & H^1_\sing(K_{m,\mathfrak{l}}, T_{f,n}) }
\]
\end{prop}

\subsection{First and Second Reciprocity
Laws}\label{ss:first-second}

By technical reason we choose an auxiliary prime
$\mathfrak{q}_0\nmid \Delta \mathfrak{n}^+$ such that
$1+\mathrm{N}(\mathfrak{q}_0)-\alpha_{\mathfrak{q}_0}(f)\in
\mathcal{O}_f^\times$.

The inclusion $t'(K^\times)\subset B'^\times \subset
\mathrm{GL}_2(\mathbb{R})$ gives an action of $K^\times$ on
$\mathbf{P}^1(\mathbb{C})-\mathbf{P}^1(\mathbb{R})$. This action has
two fixed points; we choose one and denote it by $z'$.

For $m\geq N$ and $a\in \widehat{K}^\times$ we define the Heegner
point
$$ P_m(a)=[z', \phi(a^{(\mathfrak{l})}\zeta^{(m)}\tau_N)]_N \in M_N^{[\mathfrak{l}]}(\mathbb{C}). $$
See Sections \ref{ss:n-adm} and \ref{ss:gross-theta} for the
notations $\zeta^{(m)}$ and $\tau_N$. By the theory of complex
multiplication $P_m(a)$ is defined over the ring class field
$\widetilde{K}_m$.

 We define a map
\begin{eqnarray*} \xi_{\mathfrak{q}_0}: \mathrm{Div}
(M^{[\mathfrak{l}]}_N(\widetilde{K}_m)) &\rightarrow& J^{[\mathfrak{l}]}_N(\widetilde{K}_m)_{\mathcal{O}_f} \\
P &\mapsto &
\frac{1}{1+\mathrm{N}({\mathfrak{q}_0})-\alpha_{\mathfrak{q}_0}(f)}
\mathrm{cl}((1+\mathrm{N}({\mathfrak{q}_0})-T_{\mathfrak{q}_0})P).
\end{eqnarray*} Put $$ D_m=\sum_{\sigma\in \mathrm{Gal}(\widetilde{K}_m/K_m)}\xi_{{\mathfrak{q}_0}}(P_m(1)^\sigma)
=\sum_{[a]_m\in
\mathrm{Gal}(\widetilde{K}_m/K_m)}\xi_{{\mathfrak{q}_0}}(P_m(a))\in
J^{[\mathfrak{l}]}_N(K_m)_{\mathcal{O}_f}. $$

 We
define the cohomology class $\kappa_{\mathscr{D}}(\mathfrak{l})_m$
by
$$ \kappa_\mathscr{D}(\mathfrak{l})_m := \frac{1}{\alpha_\mathfrak{p}^m} \mathrm{Kum}(D_m) \  (\mathrm{mod} \ \mathcal{I}^{[\mathfrak{l}]}_g)
\in
H^1(K_m,T_p(J^{[\mathfrak{l}]}_N)_{\mathcal{O}_f}/\mathcal{I}_g^{[\mathfrak{l}]}
)=H^1(K_m, T_{f,n}).
$$
When $m$ varies, these $\kappa_{\mathscr{D}}(\mathfrak{l})_m$ are
compatible for the corestriction maps \cite[Lemma 5.4.1]{Wang}, and
thus define an element $\kappa_{\mathscr{D}}(\mathfrak{l})$ of
$\widehat{H}^1(K_\infty, T_{f,n})$.

\begin{prop}\label{prop-sel} $($\cite[Lemma 7.16]{Longo}, \cite[Proposition
5.4.2]{Wang} $)$ $\kappa_{\mathscr{D}}(\mathfrak{l})$ belongs to
$\widehat{\mathrm{Sel}}_{\Delta \mathfrak{l}} (K_\infty, T_{f,n})$.
\end{prop}

By Proposition \ref{prop-kill} (\ref{it:free}),
$\widehat{H}^1_\sing(K_{\infty,\mathfrak{l}}, T_{f,n})$ is free of
rank one over $\mathcal{O}_{f}[[\Gamma]]/(\omega^n)$. Choosing a
base of $\widehat{H}^1_\sing(K_{\infty,\mathfrak{l}}, T_{f,n})$ we
may identify $\widehat{H}^1_\sing(K_{\infty,\mathfrak{l}}, T_{f,n})$
with $\mathcal{O}_{f}[[\Gamma]]/(\omega^n)$.

\begin{prop}\label{thm:first} $($First Reciprocity Law, \cite[Theorem 6.1.2]{Wang}$)$
Let $m\geq N \geq n$. For each $(N,n)$-admissible form
$\mathscr{D}=(\Delta, g)$ and each $n$-admissible prime
$\mathfrak{l}\nmid {\mathfrak{q}_0}\Delta$, we have
$$ \partial_\mathfrak{l}( \kappa_{\mathscr{D}}(\mathfrak{l})_m) = \theta_m(g) \in \mathcal{O}_{f,n}[\Gamma_m]
$$ up to multiplication by a unit of $\mathcal{O}_{f,n}[\Gamma_m]$.
\end{prop}
\begin{proof} By Proposition \ref{prop:comm} one has
$$ \partial_\mathfrak{l}(\kappa_{\mathscr{D}}(\mathfrak{l}))= \sum_{\sigma\in\Gamma_m} \psi_g(r_\mathfrak{l}(D_m^\sigma)) \sigma .
$$ But
$$ \psi_g(r_\mathfrak{l}(D_m^\sigma))=\sum_{[b]_m\in \mathrm{Gal}(\tilde{K}_m/K_m)}\langle g, x_m(ab)\tau_N \rangle
=\sum_{[b]_m\in \mathrm{Gal}(\tilde{K}_m/K_m)} g(x_m(ab)),
 $$ where $a\in \widehat{K}^\times$ satisfies $\sigma=\pi_m([a]_m)$.
Thus $$ \partial_\mathfrak{l}( \kappa_{\mathscr{D}}(\mathfrak{l})_m)
=\sum_{[a]_m\in G_m} g(x_m(a)) \pi_m([a]_m),$$ as desired.
\end{proof}

We fix two different $n$-admissible prime ideals $\mathfrak{l}_1$
and $ \mathfrak{l}_2$ $($$\mathfrak{l}_1,
\mathfrak{l}_2\nmid{\mathfrak{q}_0}\Delta$$)$. Then $\mathfrak{l}_1$
and $ \mathfrak{l}_2$ are inert in $K$. We fix a place
$\mathfrak{l}'_2$ of $K_m$ above $\mathfrak{l}_2$, and a place
$\tilde{\mathfrak{l}}'_2$ of $\widetilde{K}_m$ above
$\mathfrak{l}'_2$.

We have already seen that
%$J^{[\mathfrak{l}_1]}$ has smooth reduction at $\mathfrak{l}_2$,
the image of the map
$$ J_N^{[\mathfrak{l}_1]}(K_{\mathfrak{l}_2}) / \mathcal{I}_{g}^{[\mathfrak{l}_1]} \rightarrow H^1(K_{\mathfrak{l}_2}, T_{f,n})
$$ belongs to $H^1_\fin(K_{\mathfrak{l}_2}, T_{f,n})\cong
\mathcal{O}_{f,n}$, and that the reduction map
$$ J_N^{[\mathfrak{l}_1]}(K_{\mathfrak{l}_2}) / \mathcal{I}_{g}^{[\mathfrak{l}_1]} \rightarrow J^{[\mathfrak{l}_1]}(k_{\mathfrak{l}_2}) / \mathcal{I}_{g}^{[\mathfrak{l}_1]}
$$ is an isomorphism, where $k_{\mathfrak{l}_2}$ is the residue field of $K_{\mathfrak{l}_2}$.

Let $B''=B_{\Delta \mathfrak{l}_1\mathfrak{l}_2}$ be the definite
quaternion algebra with discriminant $\Delta
\mathfrak{l}_1\mathfrak{l}_2$. Then there is an isomorphism  $$\psi:
\widehat{B}^{(\mathfrak{l}_2)}_{\Delta
\mathfrak{l}_1\mathfrak{l}_2}\cong
\widehat{B}'^{(\mathfrak{l}_2)}_{\Delta\mathfrak{l}_1}.$$ Let
$\mathcal{O}_{B''_{\mathfrak{l}_1}}$ and
$\mathcal{O}_{B''_{\mathfrak{l}_2}}$ be the maximal orders of
$B''_\mathfrak{l_1}$ and $B''_\mathfrak{l_2}$ respectively. Put
$$\mathfrak{U}''=\psi((\mathfrak{U}'_{\mathfrak{n}^+,\mathfrak{p}^n})^{(\mathfrak{l}_2)})\mathcal{O}_{B''_{\mathfrak{l}_1}}^\times\mathcal{O}_{B''_{\mathfrak{l}_2}}^\times.$$

By \cite[Section 5.4]{Zhang} we have an isomorphism
$$\iota: B''^\times\backslash \widehat{B}''^{\times}/Y \mathfrak{U}''\cong \mathcal{S}_{\mathfrak{l}_2},$$  where $\mathcal{S}_{\mathfrak{l}_2}$ is the set of
supersingular points in
$J^{[\mathfrak{l}_1]}_N(k_{\mathfrak{l}_2})$. Let
$\mathbb{T}_{B_\Delta}(\mathfrak{l}_1\mathfrak{n}^+,
\mathfrak{p}^N)$ act on $\mathrm{Div}(\mathcal{S}_{\mathfrak{l}_2})$
via Picard functoriality.

The reduction $\mathrm{red}_{\tilde{\mathfrak{l}}'_2}(P_m(a))$ of
the CM point $P_m(a)$ modulo $\tilde{\mathfrak{l}}'_2$ is in
$\mathcal{S}_{\mathfrak{l}_2}$. We choose $\iota$ such that
$$ \mathrm{red}_{\tilde{\mathrm{l}}'_2} ([z', b']) = \iota (\psi^{-1}(b'^{(\mathfrak{l}_2)})) .
$$ In particular we have $$\mathrm{red}_{\tilde{\mathfrak{l}}'_2}(P_m(a)) =
\iota(x_m(a)\tau_N) .$$

So, restricting the isomorphism
$$ J_N^{[\mathfrak{l}_1]}(k_{\mathfrak{l}_2}) / \mathcal{I}_{g}^{[\mathfrak{l}_1]} \rightarrow \mathcal{O}_{f,n} $$ to $\mathcal{S}_{\mathfrak{l}_2}$ we
obtain a map
$$\gamma: \mathrm{Div}(\mathcal{S}_{\mathfrak{l}_2}) \rightarrow
\mathcal{O}_{f,n}.
$$

Write $\bar{T}$ for the image of
$T\in\mathbb{T}_{B_{\Delta}}(\mathfrak{l}_1\mathfrak{n}^+,
\mathfrak{p}^n)$ into
$\mathbb{T}_{B_{\Delta}}(\mathfrak{l}_1\mathfrak{n}^+,
\mathfrak{p}^n)/\mathcal{I}_{g}^{[\mathfrak{l}_1]}$.

\begin{prop}\label{prop:relations} $($\cite[Lemma 7.17]{Longo}$)$ For $x\in \mathrm{Div}(\mathcal{S}_{\mathfrak{l}_2})$
the following relations hold:
\begin{enumerate}
\item For $\mathfrak{q}\nmid \Delta\mathfrak{n}^+\mathfrak{l}_1$, one has
$\gamma(T_\mathfrak{q}x)=\bar{T}_\mathfrak{q}\gamma(x)$.
\item For $\mathfrak{q}| \Delta \mathfrak{n}^+\mathfrak{l}_1$, one has
$\gamma(U_\mathfrak{q}x)=\bar{U}_\mathfrak{q}\gamma(x)$.
\item
$\gamma(T_{\mathfrak{l}_2}x)=\bar{T}_{\mathfrak{l}_2}\gamma(x)$.
\item $\gamma(\mathrm{Frob}_{\mathfrak{l}_2}(x))= \epsilon_{\mathfrak{l}_2}
\gamma(x)$, where $\mathrm{Frob}_{\mathfrak{l}_2}$ is the absolute
Frobenius of $F$ at $\mathfrak{l}_2$.
\end{enumerate}
\end{prop}

The relation between $\gamma$ and the system
$\{\kappa_{\mathscr{D}}(\mathfrak{l}_1)_m: m\geq N\geq n \}$ is
given by the following.

\begin{prop} \label{prop:second-recip-law} If $(\Delta,g)$ is an
$(N,n)$-admissible form, and if $m\geq N$, then
\begin{eqnarray*}v_{\mathfrak{l}_2}(\kappa_{\mathscr{D}}(\mathfrak{l}_1)_m)
&=& \frac{1}{\alpha_\mathfrak{p}^m} \sum_{[a]_m\in G_m} \gamma\circ
\iota(x_m(a)\tau_N)\pi_m([a]_m)
\end{eqnarray*} in $\mathcal{O}_{f,n}[\Gamma_m]$.
\end{prop}

Proposition \ref{prop:second-recip-law} is more or less contained in
\cite{Longo,Wang}, but it is not stated in the above form.

\begin{proof}
All primes of $K_m$ above $\mathfrak{l}_2$ are $\{\sigma
\mathfrak{l}'_2: \sigma\in\Gamma_m\}$. So
\begin{eqnarray*}v_{\mathfrak{l}_2}(\kappa_{\mathscr{D}}(\mathfrak{l}_1)_m)
&=& \sum_{\sigma\in \Gamma_m} v_{\sigma \mathfrak{l}'_2}
(\kappa_{\mathscr{D}}(\mathfrak{l}_1)_m) \\
&=& \sum_{\sigma\in \Gamma_m} v_{ \mathfrak{l}'_2}
(\kappa_{\mathscr{D}}(\mathfrak{l}_1)_m^{\sigma^{-1}})\sigma \\
&=& \frac{1}{\alpha_\mathfrak{p}^m} \sum_{[a]_m\in G_m}  v_{
\tilde{\mathfrak{l}}'_2}(P_m(a))\pi_m([a]_m).
\end{eqnarray*} Note that the reduction of $P_m(a)$ modulo
$\widetilde{\mathfrak{l}}'_2$ lies in $\mathcal{S}_2$. Thus
$$ v_{
\tilde{\mathfrak{l}}'_2}(P_m(a)) = \gamma
(\mathrm{red}_{\tilde{\mathfrak{l}}'_2}(P_m(a))) = \gamma\circ \iota
(x_m(a)\tau_N),
$$ as wanted.
\end{proof}

\begin{cor}\label{cor:gamma-neq-0} If there exists $m$ such that
$v_{\mathfrak{l}_2}(\kappa_{\mathscr{D}}(\mathfrak{l}_1)_m)\neq 0$,
then $\gamma\neq 0$.
\end{cor}

\begin{prop}\label{prop:surj-Ihara} $($\cite[Lemma 7.20]{Longo}$)$ If Ihara's Lemma holds, then $\gamma$ is surjective.
\end{prop}

By Proposition \ref{prop-kill} (\ref{it:free}),
$\widehat{H}^1_\fin(K_{\infty,\mathfrak{l}_1}, T_{f,n})$ and
$\widehat{H}^1_\fin(K_{\infty,\mathfrak{l}_2}, T_{f,n})$ are free of
rank $1$ over $\mathcal{O}_{f}[[\Gamma]]/(\omega^r)$. We may
identify both $\widehat{H}^1_\fin(K_{\infty,\mathfrak{l}_1},
T_{f,n})$ and $\widehat{H}^1_\fin(K_{\infty,\mathfrak{l}_2},
T_{f,n})$ with $\mathcal{O}_{f}[[\Gamma]]/(\omega^n)$.

\begin{prop} $($Second Reciprocity Law, \cite[Theorem 6.6]{Wang}$)$ If Ihara's Lemma holds, then
$$v_{\mathfrak{l}_2}(\kappa_{\mathscr{D}}(\mathfrak{l}_1)_m) =
v_{\mathfrak{l}_1}(\kappa_{\mathscr{D}}(\mathfrak{l}_2)_m) $$ up to
multiplication by a unit of $\mathcal{O}_{f,n}[\Gamma_m]$.
\end{prop}

\subsection{A weaker version of Second Reciprocity Law}

One expects to show that
$$ v_{\mathfrak{l}_2}(\kappa_{\mathscr{D}}(\mathfrak{l}_1)_m) =
v_{\mathfrak{l}_1}(\kappa_{\mathscr{D}}(\mathfrak{l}_2)_m) $$
without using Ihara's lemma. But we can only prove a weaker result.
We will deduce from the first reciprocity law and Tate duality that
they coincide with each other after multiplying $\theta_m(g)$.

Let $\tau$ be a complex conjugation which depends on a choice of
embedding of the algebraic closure of $\overline{F}$ in
$\mathbb{C}$. For $\sigma \in \Gamma_m$ we have
$\tau\sigma=\sigma^{-1}\tau$. The homomorphism $\sigma\mapsto
\sigma^{-1}$ of $\Gamma_m$ induces an involution $\iota$ on
$\mathcal{O}_{f,r}[\Gamma_m]$. Then $\tau$ acts on
$\mathcal{O}_{f,r}[\Gamma_m]$ as $\iota$.

For each $i\in\{1,2\}$, as $\mathfrak{l}_i$ splits completely in
$K_m$ \cite[Lemma 2.4.2]{Wang}, the number of places of $K_m$ above
$\mathfrak{l}_i$ are $[K_m:K]$. Fix a place $\mathfrak{l}'_i$ of
$K_m$ above $\mathfrak{l}_i$. Then all places of $K_m$ above
$\mathfrak{l}_i$ are $\{\sigma \mathfrak{l}'_i: \sigma\in
\mathrm{Gal}(K_m/K)\}$. Note that $\tau$ permutes $\{\sigma
\mathfrak{l}'_i: \sigma\in \mathrm{Gal}(K_m/K)\}$.

Note that
$$ H^1_\fin(K_{m,\mathfrak{l}_i}, T_{f,n})\cong H^1_\fin(K_{m,\tau\mathfrak{l}'_i}, T_{f,n})\otimes_{\mathcal{O}_f}\mathcal{O}_f[\Gamma_m] $$ and
$$H^1_\sing(K_{m,\mathfrak{l}_i}, T_{f,n})\cong H^1_\sing(K_{m,\mathfrak{l}'_i},
T_{f,n})\otimes_{\mathcal{O}_f}\mathcal{O}_f[\Gamma_m].$$

Both $H^1_\fin(K_{m,\tau\mathfrak{l}'_i}, T_{f,n})$ and
$H^1_\sing(K_{m,\mathfrak{l}'_i}, T_{f,n})$ are isomorphic to
$\mathcal{O}_{f,n}$. We choose generators $c_{\tau\mathfrak{l}'_i}$
and $d_{\mathfrak{l}'_i}$ of $H^1_\fin(K_{m,\tau\mathfrak{l}'_i},
T_{f,n})$ and $H^1_\sing(K_{m,\mathfrak{l}'_i}, T_{f,n})$ such that
$\langle  c_{\tau\mathfrak{l}'_1}, \tau
d_{\mathfrak{l}'_1}\rangle_{\tau \mathfrak{l}'_1}=1$ and $\langle
\tau c_{\tau\mathfrak{l}'_2},   d_{\mathfrak{l}'_2}\rangle_{
\mathfrak{l}'_2}=1$.

\begin{lem}\label{lem:dual} For each $\sigma\in\Gamma_m$ we have
$$ \langle \sigma c_{\tau\mathfrak{l}'_1}, \sigma \tau
d_{\mathfrak{l}'_1}\rangle_{\sigma\tau \mathfrak{l}'_1}=\langle
\sigma \tau c_{\tau\mathfrak{l}'_2},  \sigma
d_{\mathfrak{l}'_2}\rangle_{ \sigma \mathfrak{l}'_2}=1. $$
\end{lem}
\begin{proof} Let $\mathrm{Res}:
H^1(K_{\mathfrak{l}_1}, T_{f,n})\rightarrow
H^1(K_{m,\mathfrak{l}_1}, T_{f,n})$ and
$\mathrm{Cores}:H^1(K_{m,\mathfrak{l}_1}, T_{f,n})\rightarrow
H^1(K_{\mathfrak{l}_1}, T_{f,n})$ be the restriction map and the
corestriction map respectively.

As $\sum\limits_{\gamma\in\Gamma_m}\gamma \tau d_{\mathfrak{l}'_1} $
is fixed by $\Gamma_m$, we have
$\sum\limits_{\gamma\in\Gamma_m}\gamma \tau
d_{\mathfrak{l}'_1}=\mathrm{Res} (x)$ for some $x\in
H^1(K_{\mathfrak{l}_1}, T_{f,n})$. Then $$ \langle \sigma
c_{\tau\mathfrak{l}'_1}, \sigma \tau
d_{\mathfrak{l}'_1}\rangle_{\sigma\tau \mathfrak{l}'_1} = \langle
\sigma c_{\tau\mathfrak{l}'_1}, \sum_{\gamma\in \Gamma_m}
\gamma\sigma \tau d_{\mathfrak{l}'_1}\rangle_{ \mathfrak{l}_1} =
\langle \sigma c_{\tau\mathfrak{l}'_1},
\mathrm{Res}(x)\rangle_{\mathfrak{l}_1} = \langle
\mathrm{Cores}(\sigma c_{\tau\mathfrak{l}'_1}),
x\rangle_{\mathfrak{l}_1}.
$$ As $\mathrm{Cores}(\sigma
c_{\tau\mathfrak{l}'_1})=\mathrm{Cores}(c_{\tau\mathfrak{l}'_1})$,
we obtain $$\langle \sigma c_{\tau\mathfrak{l}'_1}, \sigma \tau
d_{\mathfrak{l}'_1}\rangle_{\sigma\tau \mathfrak{l}'_1}=\langle
c_{\tau\mathfrak{l}'_1}, \tau d_{\mathfrak{l}'_1}\rangle_{\tau
\mathfrak{l}'_1}=1.$$

The proof of $$ \langle \sigma \tau c_{\tau\mathfrak{l}'_2}, \sigma
d_{\mathfrak{l}'_2}\rangle_{ \sigma \mathfrak{l}'_2}=1  $$ is
similar.
\end{proof}

Proposition \ref{thm:first} says that there exist two units $u_{1}$
and $ u_{2}$ in $ \mathcal{O}_{f,r}[\Gamma_m]$ such that
\begin{eqnarray*}
\partial_{\mathfrak{l}_i}
(\kappa_\mathscr{D}(\mathfrak{l}_i)_m) &=& u_{i}  \theta_m(g) \cdot
d_{\mathfrak{l}'_i} . \end{eqnarray*} Let $\theta_1$ and $\theta_2$
be the two elements in $ \mathcal{O}_{f,n}[\Gamma_m]$ such that
\begin{eqnarray*}
 v_{\mathfrak{l}_2} (\kappa_\mathscr{D}(\mathfrak{l}_1)_m) = \theta_1 c_{\tau \mathfrak{l}'_2}
\hskip 10pt   \text{ and } \hskip 10pt v_{\mathfrak{l}_1}
(\kappa_\mathscr{D}(\mathfrak{l}_2)_m) = \theta_2
c_{\tau\mathfrak{l}'_1} . \end{eqnarray*}

\begin{thm} \label{prop:theta}
We have
\begin{equation} \label{eq:key-formula}
\theta_m(g) (u_2\theta_1+u_1\theta_2) = 0
\end{equation}
in $\mathcal{O}_{f,n}[\Gamma_m]$.
\end{thm}
\begin{proof} Note that $T_{f,n}$ is self dual, so we can form the
local Tate pairing $\langle\cdot, \cdot\rangle_v$ on $H^1(K_{m,v},
T_{f,n})$ for each place $v$ of $K_m$.

For any $c_1, c_2\in H^1(K_m, T_{f,n})$ and each place $v$ of $K_m$
we write $\langle c_1, c_2\rangle_v=\langle \res_v(c_1),
\res_v(c_2)\rangle_v$. Then $\sum_v \langle c_1, c_2\rangle_v =0$.
We apply this to $c_1=\tau \kappa_{\mathscr{D}}(\mathfrak{l}_1)_m $
and $c_2=\gamma\kappa_{\mathscr{D}}(\mathfrak{l}_2)_m$ with
$\gamma\in\Gamma_m$.

By Lemma \ref{lem:preserve }, $c_1\in
 {\mathrm{Sel}}_{\Delta \mathfrak{l}_1} (K_m, T_{f,n})$, and $c_2\in {\mathrm{Sel}}_{\Delta \mathfrak{l}_2} (K_m,
 T_{f,n})$. So, when $v$ is not above $\mathfrak{l}_1$ or
$\mathfrak{l}_2$ we have $$\langle
\tau\kappa_{\mathscr{D}}(\mathfrak{l}_1),
\gamma\kappa_{\mathscr{D}}(\mathfrak{l}_2)\rangle_v =0.$$ Hence,
$$ \sum_{\sigma\in \Gamma_m} ( \langle\tau\kappa_{\mathscr{D}}(\mathfrak{l}_1), \gamma\kappa_{\mathscr{D}}(\mathfrak{l}_2)\rangle_{\sigma \mathfrak{l}'_1} +
\langle\tau\kappa_{\mathscr{D}}(\mathfrak{l}_1),
\gamma\kappa_{\mathscr{D}}(\mathfrak{l}_2)\rangle_{\sigma
\mathfrak{l}'_2} ) =0. $$

We write $$u_i\theta_m(g)=\sum_{\sigma\in\Gamma_m}
a_{i,\sigma}\sigma, \hskip 10pt a_{i,\sigma}\in\mathcal{O}_{f,n}$$
and
$$\theta_i=\sum_{\sigma\in\Gamma_m}=\sum_{\sigma\in\Gamma_m}b_{i,\sigma}\sigma, \hskip 10pt b_{i,\sigma}\in\mathcal{O}_{f,n}. $$

Then $$\partial_{\mathfrak{l}_1} (
\tau\kappa_{\mathscr{D}}(\mathfrak{l}_1) )=
\iota(u_1\theta_m(g))\tau d_{\mathfrak{l}'_1}
=\sum_{\sigma\in\Gamma_m} a_{1,\sigma^{-1}}\sigma \tau
d_{\mathfrak{l}'_1}
$$ and
$$ v_{\mathfrak{l}_1} (\gamma\kappa_{\mathscr{D}}(\mathfrak{l}_2)) = \gamma \theta_2 c_{\tau \mathfrak{l}'_1}
=\sum_{\sigma\in\Gamma_m} b_{2,\sigma\gamma^{-1}}\sigma c_{\tau
\mathfrak{l}'_1}. $$ By Lemma \ref{lem:dual} we have
\begin{eqnarray*} \langle\tau\kappa_{\mathscr{D}}(\mathfrak{l}_1),
\gamma\kappa_{\mathscr{D}}(\mathfrak{l}_2)\rangle_{\sigma \tau
\mathfrak{l}'_1} &= & \langle
\partial_{\mathfrak{l}_1}(\tau\kappa_{\mathscr{D}}(\mathfrak{l}_1)),
v_{\mathfrak{l}_1}(\gamma\kappa_{\mathscr{D}}(\mathfrak{l}_2))\rangle_{\sigma
\tau \mathfrak{l}'_1} \\
& =& \langle  a_{1,\sigma^{-1}}\sigma \tau d_{\mathfrak{l}'_1},
b_{2,\sigma\gamma^{-1}}\sigma c_{\tau \mathfrak{l}'_1}
\rangle_{\sigma \tau
\mathfrak{l}'_1}=a_{1,\sigma^{-1}}b_{2,\sigma\gamma^{-1}} .
\end{eqnarray*} Hence,
$$\sum_{\sigma\in \Gamma_m}
\langle\tau\kappa_{\mathscr{D}}(\mathfrak{l}_1),
\gamma\kappa_{\mathscr{D}}(\mathfrak{l}_2)\rangle_{\sigma
\mathfrak{l}'_1} = \sum_{\sigma\in \Gamma_m}
\langle\tau\kappa_{\mathscr{D}}(\mathfrak{l}_1),
\gamma\kappa_{\mathscr{D}}(\mathfrak{l}_2)\rangle_{\sigma \tau
\mathfrak{l}'_1} = \sum_{\sigma\in\Gamma_m }
a_{1,\sigma^{-1}}b_{2,\sigma\gamma^{-1}}  .$$ Similarly,
$$\sum_{\sigma\in \Gamma_m}
\langle\tau\kappa_{\mathscr{D}}(\mathfrak{l}_1),
\gamma\kappa_{\mathscr{D}}(\mathfrak{l}_2)\rangle_{\sigma
\mathfrak{l}'_1} = \sum_{\sigma\in\Gamma_m }
b_{1,\sigma^{-1}}a_{2,\sigma\gamma^{-1}}  .$$ Therefore,
$$ \sum_{\sigma\in\Gamma_m }
(a_{1,\sigma^{-1}}b_{2,\sigma\gamma^{-1}}+b_{1,\sigma^{-1}}a_{2,\sigma\gamma^{-1}})=0.
$$
This sum is just the coefficient of $\gamma^{-1}$ in the left hand
of (\ref{eq:key-formula}). This proves (\ref{eq:key-formula}).
\end{proof}

For any $a\in \mathcal{O}_{f,n}$, if $a=u\omega^s$ with $u$ a unit
in $\mathcal{O}_{f,n}$, and $s\in \{0,1,\cdots, n\}$, we put
$\mathrm{val}_{\omega}(a)=s$.

Let $\varphi: \mathcal{O}_{f,n}[\Gamma_m]\rightarrow
\mathcal{O}_{f,n}$ be a homomorphism. For each
$\theta\in\mathcal{O}_{f,n}[\Gamma_m]$ we put
$\ord_{\varphi}(\theta):=\mathrm{val}_{\omega}(\varphi(\theta))$.

For each element $x$ of ${H}^1(K_\infty, T_{f,n})$ we write
$\varphi(x)$ for its image in $${H}^1(K_\infty,
T_{f,n})\otimes_\varphi \mathcal{O}_{f,n}\cong \mathcal{O}_{f,n},$$
and put $$\ord_\varphi(x)=\mathrm{val}_\omega ( \varphi(x) ).$$ Then
we have $$
\ord_\varphi(v_{\mathfrak{l}_2}(\kappa_{\mathscr{D}}(\mathfrak{l}_1)_m))=\ord_\varphi
(\theta_1)$$ and $$
\ord_\varphi(v_{\mathfrak{l}_1}(\kappa_{\mathscr{D}}(\mathfrak{l}_2)_m))=\ord_\varphi
(\theta_2).$$

\begin{cor} \label{cor:theta-2}
If $\varphi: \mathcal{O}_{f,n}[\Gamma_m]\rightarrow
\mathcal{O}_{f,n}$ is a homomorphism such that
$$ \ord_{\varphi}(\partial_{\mathfrak{l}_1}(\kappa_{\mathscr{D}}(\mathfrak{l}_1)_m))  +
\ord_\varphi(v_{\mathfrak{l}_2}(\kappa_{\mathscr{D}}(\mathfrak{l}_1)_m))<n
,$$ then
$$\ord_\varphi(v_{\mathfrak{l}_2}(\kappa_{\mathscr{D}}(\mathfrak{l}_1)_m)) =
\ord_\varphi(v_{\mathfrak{l}_1}(\kappa_{\mathscr{D}}(\mathfrak{l}_2)_m)).$$
\end{cor}
\begin{proof} By Proposition \ref{prop:theta} we have
$$
\varphi( \theta_m(g))
(\varphi(u_2)\varphi(\theta_1)+\varphi(u_1)\varphi(\theta_2)) = 0 .
$$ Note that $ \varphi(u_1) $ and $\varphi(u_2)$ are units of
$\mathcal{O}_{f,n}$.

Our assertion comes from  the following two trivial properties of
$\mathrm{val}_\omega$:

$\bullet$ Either if
$\mathrm{val}_\omega(a)+\mathrm{val}_\omega(b)<n$, or if
$\mathrm{val}_\omega(ab)<n$, then
$\mathrm{val}_\omega(a)+\mathrm{val}_\omega(b)=\mathrm{val}_\omega(ab)$.

$\bullet$ If $a+b=0$, then
$\mathrm{val}_\omega(a)=\mathrm{val}_\omega(b)$.
\end{proof}

\subsection{Admissible form}

\begin{prop}\label{thm:second} Let $(\Delta,g)$ be an
$(N,n)$-admissible form. If $\mathfrak{l}_1$ and $\mathfrak{l}_2$
$($$\mathfrak{l}_1,\mathfrak{l}_2\nmid\mathfrak{q}_0\Delta$$)$ are
$n$-admissible prime ideals, and if $m\geq N$ is an integer such
that $v_{\mathfrak{l}_2}(\kappa_{\mathscr{D}}(\mathfrak{l}_1)_m)
\neq 0$, then there exists a nonnegative integer $n_0<n$ and an $(N,
n-n_0)$-admissible form $(\Delta \mathfrak{l}_1\mathfrak{l}_2, g'')$
satisfying the following.
\begin{enumerate}
\item For any homomorphism $\varphi: \mathcal{O}_{f,n}[\Gamma_m]\rightarrow
\mathcal{O}_{f,n}$ we have $n_0\leq
\ord_\varphi(v_{\mathfrak{l}_2}(\kappa_{\mathscr{D}}(\mathfrak{l}_1)_m))$.
\item We have
$$ v_{\mathfrak{l}_2}(\kappa_{\mathscr{D}}(\mathfrak{l}_1)_m)
= \omega^{n_0} \theta_m(g'') \in \mathcal{O}_{f,n}[\Gamma_m]
$$ up to multiplication by a unit of $\mathcal{O}_{f,n}[\Gamma_m]$.
\end{enumerate}
\end{prop}

Here, $\theta_m(g'')$ is in $\mathcal{O}_{f,n-n_0}[\Gamma_m]$. The
homomorphism
\begin{eqnarray*} \mathcal{O}_{f,n}[\Gamma_m] &\xrightarrow{\times
\omega^{n_0}} &
\mathcal{O}_{f,n}[\Gamma_m], \\
\sum_{\sigma\in\Gamma_m}a_\sigma \sigma &\mapsto&
\sum_{\sigma\in\Gamma_m}\omega^{n_0} a_\sigma \sigma
\end{eqnarray*} annihilates
$\omega^{n-n_0}\mathcal{O}_{f,n}[\Gamma_m]$ and thus induces an
homomorphism
$$ \mathcal{O}_{f, n-n_0}[\Gamma_m] \xrightarrow{\times \omega^{n_0}} \mathcal{O}_{f, n}[\Gamma_m]. $$

\begin{proof}
Let $n_0$ be the largest integer such that $\mathrm{Im}(\gamma)\in
\omega^{n_0} \mathcal{O}_{f,n} $. By Corollary \ref{cor:gamma-neq-0}
we have $$
v_{\mathfrak{l}_2}(\kappa_{\mathscr{D}}(\mathfrak{l}_1)_m)\in
\omega^{n_0}\mathcal{O}_{f,n}[\Gamma_m].$$ Thus for any homomorphism
$\varphi: \mathcal{O}_{f,n}[\Gamma]\rightarrow \mathcal{O}_{f,n}$ we
have $$\varphi
(v_{\mathfrak{l}_2}(\kappa_{\mathscr{D}}(\mathfrak{l}_1)_m))\in
\omega^{n_0}\mathcal{O}_{f,n}$$ yielding
$$\ord_\varphi(v_{\mathfrak{l}_2}(\kappa_{\mathscr{D}}(\mathfrak{l}_1)_m))\geq
n_0.$$

Let $\widetilde{\gamma}$ be any map $$\widetilde{\gamma}:
\mathrm{Div}(\mathcal{S}_{\mathfrak{l}_2}) \rightarrow
\mathcal{O}_{f,n}$$ such that
$\gamma=\omega^{n_0}\widetilde{\gamma}$. Let $\gamma'$ be the
composition
$$ \mathrm{Div}(\mathcal{S}_{\mathfrak{l}_2})
\xrightarrow{\widetilde{\gamma}} \mathcal{O}_{f,n} \rightarrow
\mathcal{O}_{f,n-n_0} ,
$$ where $\mathcal{O}_{f,n} \rightarrow
\mathcal{O}_{f,n-n_0}$ is the natural quotient map.

If $\mathfrak{q}\nmid
\Delta\mathfrak{l}_1\mathfrak{l}_2\mathfrak{n}^+$, from
$\gamma(T_\mathfrak{q}x)-\bar{T}_\mathfrak{q}\gamma(x)=0$, we get
$$ \widetilde{\gamma}(T_\mathfrak{q}x)-\bar{T}_\mathfrak{q}\widetilde{\gamma}(x)
\in \omega^{n-n_0}\mathcal{O}_{f,n}. $$ It follows that
$$\gamma'(T_\mathfrak{q}x)-\bar{T}_\mathfrak{q}\gamma'(x)=0.$$  The
same argument shows that, if $\mathfrak{q}|
\Delta\mathfrak{n}^+\mathfrak{l}_1$, then
$$\gamma'(U_\mathfrak{q}x)=\bar{U}_\mathfrak{q}\gamma'(x).$$ In
particular, $\gamma'(U_{\mathfrak{l}_1} x)=
\epsilon_{\mathfrak{l}_1}\gamma'(x)$. Similarly,
$\gamma'(\mathrm{Frob}_{\mathfrak{l}_2}(x))=
\epsilon_{\mathfrak{l}_2} \gamma'(x)$. By \cite[Section 9]{Carayol}
we have $U_{\mathfrak{l}_2}=\mathrm{Frob}_{\mathfrak{l}_2}$ on
$\mathrm{Div}(\mathcal{S}_{\mathfrak{l}_2})$. Hence,
$$\gamma'(U_{\mathfrak{l}_2}
x)=\gamma'(\mathrm{Frob}_{\mathfrak{l}_2}
x)=\epsilon_{\mathfrak{l}_2}\gamma'(x).$$

Let $$g''\in
S^{B_\Delta\mathfrak{l}_1\mathfrak{l}_2}_2(\mathfrak{U}'',
\mathcal{O}_{f,n-n_0})^Y$$ be the function such that
$\psi_{g''}=\gamma'$. Since $\gamma'$ is Hecke equivariant,
$(\Delta\mathfrak{l}_1\mathfrak{l}_2, g'')$ is an
$(N,n-n_0)$-admissible form. By (\ref{eq:psi-g}) we have
$$ g''(x_m(a))=\gamma'(x_m(a)\tau_N) .$$

By Proposition \ref{prop:second-recip-law} we have
\begin{eqnarray*}v_{\mathfrak{l}_2}(\kappa_{\mathscr{D}}(\mathfrak{l}_1)_m)
&=& \frac{1}{\alpha_\mathfrak{p}^m} \sum_{[a]_m\in G_m}
\gamma(x_m(a)\tau_N)\pi_m([a]_m)\\
&=& \frac{1}{\alpha_\mathfrak{p}^m} \sum_{[a]_m\in G_m} \omega^{n_0} {\gamma}'(x_m(a)\tau_N)\pi_m([a]_m) \\
&=& \frac{\omega^{n_0}}{\alpha_\mathfrak{p}^m} \sum_{[a]_m\in G_m}
  g''(x_m(a))\pi_m([a]_m) \hskip 5pt = \hskip 5pt  \omega^{n_0}\theta_m(g''),
\end{eqnarray*} as desired. \end{proof}

\begin{rem} Proposition \ref{prop:surj-Ihara} says that, if Ihara's lemma holds, then $n_0=0$.
\end{rem}

We can strengthen the statement of Corollary \ref{cor:theta-2}.
Though it will not be used in the next section, we give it below for
its own interest.

\begin{thm}\label{thm:theta}
Assume $(\mathrm{CR}^+)$ and $(\mathfrak{n}^+\text{-}\mathrm{DT})$
hold. If there exists a homomorphism
$$ \varphi:\mathcal{O}_{f,n}[\Gamma_m]\rightarrow
\mathcal{O}_{f,n} $$ such that
\begin{equation} \label{eq:condition}
\ord_{\varphi}(\partial_{\mathfrak{l}_1}(\kappa_{\mathscr{D}}(\mathfrak{l}_1)_m))
+
\ord_\varphi(v_{\mathfrak{l}_2}(\kappa_{\mathscr{D}}(\mathfrak{l}_1)_m))<n
,\end{equation} then
$$v_{\mathfrak{l}_2}(\kappa_{\mathscr{D}}(\mathfrak{l}_1)_m) =
v_{\mathfrak{l}_1}(\kappa_{\mathscr{D}}(\mathfrak{l}_2)_m)$$ up to
multiplication by elements of $\mathcal{O}_{f,n}^\times$ and
$\Gamma_m$.
\end{thm}
\begin{proof} By Corollary \ref{cor:theta-2} it follows from
(\ref{eq:condition}) that
$$\ord_\varphi(v_{\mathfrak{l}_1}(\kappa_{\mathscr{D}}(\mathfrak{l}_2)))=
\ord_\varphi(v_{\mathfrak{l}_2}(\kappa_{\mathscr{D}}(\mathfrak{l}_1)))
< n.$$

Let $n_0$ and $g''$ be as in Proposition \ref{thm:second}. Then
$$n_0\leq \ord_\varphi(v_{\mathfrak{l}_2}(\kappa_{\mathscr{D}}(\mathfrak{l}_1)))$$
and
$$  v_{\mathfrak{l}_2}(\kappa_{\mathscr{D}}(\mathfrak{l}_1))
= \omega^{n_0} \theta(g'') \in \mathcal{O}_{f,n}[[\Gamma]]
$$ up to multiplication by a unit of $\mathcal{O}_{f,n}[[\Gamma]]$.
Exchanging $\mathfrak{l}_1$ and $\mathfrak{l}_2$, by Proposition
\ref{thm:second} there exists a nonnegative integer $$n'_0\leq
\ord_\varphi(v_{\mathfrak{l}_1}(\kappa_{\mathscr{D}}(\mathfrak{l}_2)))$$
and an $(N, n-n'_0)$-admissible form $(\Delta
\mathfrak{l}_1\mathfrak{l}_2, h'')$ such that $$
v_{\mathfrak{l}_1}(\kappa_{\mathscr{D}}(\mathfrak{l}_2)) =
\omega^{n'_0} \theta(h'') \in \mathcal{O}_{f,n}[[\Gamma]]
$$ up to multiplication by a unit of $\mathcal{O}_{f,n}[[\Gamma]]$.

Without loss of generality we may assume that $ n_{0} \leq n'_{0}$.
When $(\mathrm{CR}^+)$ and $(\mathfrak{n}^+\text{-}\mathrm{DT})$
hold, the multiplicity one theorem holds \cite[Theorem 9.1.1]{Wang}
from which we obtain
$$h'' \equiv g''\ (\mathrm{mod}\ \omega^{n_0}).$$ So
$$\omega^{n_0}\theta_m(h'')=\omega^{n_0}\theta_m(g'')$$ in
$\mathcal{O}_{f,n}[\Gamma_m]$. It follows that \begin{eqnarray*}
\ord_\varphi
(v_{\mathfrak{l}_1}(\kappa_{\mathscr{D}}(\mathfrak{l}_2))) & = &
n'_0+ \ord_\varphi(\theta_m(h'')) \\
&=& (n'_0-n_0)+ (n_0+ \ord_\varphi\theta_m(g'')) \ \ = \ \
(n'_0-n_0) + \ord_\varphi
(v_{\mathfrak{l}_2}(\kappa_{\mathscr{D}}(\mathfrak{l}_1))) .
\end{eqnarray*} Since $$\ord_\varphi(v_{\mathfrak{l}_1}(\kappa_{\mathscr{D}}(\mathfrak{l}_2)))=
\ord_\varphi(v_{\mathfrak{l}_2}(\kappa_{\mathscr{D}}(\mathfrak{l}_1)))
< n, $$ we obtain $n'_0=n_0$ yielding our conclusion.  \end{proof}

\section{Proof of Theorem \ref{thm:main}} \label{sec:BD-argument}

Let $\varphi: \mathcal{O}_f[[\Gamma]]\rightarrow \mathcal{O}_f$ be a
homomorphism from $\mathcal{O}_f[[\Gamma]]$ to the ring of integers
in a finite extension of $\BQ_p$. Let $\omega$ be a uniformizing
element of $\mathcal{O}_f$. For each positive integer $r$, let
$\varphi_r$ be the composition
$$\mathcal{O}_f[[\Gamma]]\xrightarrow{\varphi}
\mathcal{O}_f\rightarrow
\mathcal{O}_{f,r}=\mathcal{O}_f/(\omega^r).$$

We write $\mathrm{ord}$ for the valuation of $\mathcal{O}_f$ whose
value on $\omega$ is $1$.

\begin{thm}\label{thm:auxi} Let $N\geq r$ be two positive integers.
If $\mathscr{D}=(\Delta,g)$ is an $(N,r)$-admissible form, and if
$t_{\varphi,g} :=\ord (\varphi_r ( \theta(g)))$ satisfies
$2t_{\varphi,g}\leq r$. Then for each positive integer $n\leq
r-t_{\varphi,g}$, we have
\begin{equation}\label{eq:key-form}
\mathrm{length}_{\mathcal{O}_f}(\mathrm{Sel}_\Delta(K_\infty,
A_{f,n})^\vee\otimes_{\varphi}\mathcal{O}_f)\leq 2
t_{\varphi,g}.\end{equation}
\end{thm}

We fix an integer $m\geq N$ such that $\varphi_N$ factors through
$\mathcal{O}_{f,N}[\Gamma_m]$. Then $\varphi_r$ factors through
$\mathcal{O}_{f,r}[\Gamma_m]$. So
$\varphi_r(\theta(g))=\varphi_r(\theta_m(g))$ and $t_{\varphi,g}
=\ord (\varphi_r ( \theta_m(g)))$.

We prove (\ref{eq:key-form}) by induction on $t_{\varphi,g}$.

First we assume $(\mathrm{CR}^+)$, $(\mathrm{PO})$ and
$(\mathfrak{n}^+\text{-}\mathrm{min})$ hold. By Proposition
\ref{prop-free} there exists a finite set $S$ of $r$-admissible
prime ideals such that $\widehat{\mathrm{Sel}}^S_\Delta(K_\infty,
T_{f,r})\otimes_\varphi \mathcal{O}_f$ is free over
$\mathcal{O}_f/(\omega^{r})$. We fix such a set $S$. Let $$s_1,
\cdots, s_d \hskip 10pt
(d=\mathrm{rank}_{\mathcal{O}_f/(\omega^{r})}\widehat{\mathrm{Sel}}^S_\Delta(K_\infty,
T_{f,r})\otimes_\varphi \mathcal{O}_f)$$ be a basis of
$\widehat{\mathrm{Sel}}^S_\Delta(K_\infty, T_{f,r})\otimes_\varphi
\mathcal{O}_f$ over $\mathcal{O}_f/(\omega^{r})$. For any element
$\sum_i a_i s_i$ in $ \widehat{\mathrm{Sel}}^S_\Delta(K_\infty,
T_{f,r})\otimes_\varphi \mathcal{O}_f$ we define $$\ord (\sum_i a_i
s_i):=\min \{\ord (a_i): i=1,\cdots,d\}\in \{0,1 , \cdots, r\}.$$
Note that this does not depend on the choice of the basis $\{s_i:
i=1,\cdots,d\}$.

For any $r$-admissible prime ideal $\mathfrak{l}\notin S$,
considering
$\kappa_\varphi(\mathfrak{l})=\varphi(\kappa_{\mathscr{D}}(\mathfrak{l}))$
as an element of $\widehat{\mathrm{Sel}}^S_\Delta(K_\infty,
T_{f,r})\otimes_\varphi \mathcal{O}_f$, we put
$e_{\mathfrak{l}}=\ord \kappa_\varphi(\mathfrak{l})$. By Proposition
\ref{thm:first}, the first reciprocity law, we have
$e_\mathfrak{l}\leq t_{\varphi,g}$.

Then there exists
$$\tilde{\kappa}'(\mathfrak{l})\in
\widehat{\mathrm{Sel}}^S_\Delta(K_\infty,
T_{f,r})\otimes_\varphi\mathcal{O}_f$$ such that
$\omega^{e_\mathfrak{l}}\tilde{\kappa}'(\mathfrak{l}) =
\kappa_\varphi(\mathfrak{l})$.

The quotient map $T_{f,r}\rightarrow T_{f,n}$ induces a homomorphism
$$\widehat{\mathrm{Sel}}^S_\Delta(K_\infty,
T_{f,r})\otimes_\varphi\mathcal{O}_f\rightarrow
\widehat{H}^1(K_\infty, T_{f,n})\otimes_\varphi \mathcal{O}_f.$$

\begin{lem}\label{lem:basic} Let $\kappa'(\mathfrak{l})$ be the image of $\tilde{\kappa}'(\mathfrak{l})$ in
$ \widehat{H}^1(K_\infty, T_{f,n})\otimes_\varphi \mathcal{O}_f$.
\begin{enumerate}
\item \label{it:basic-a}$\ord \ \kappa'(\mathfrak{l})=0$
\item \label{it:basic-b} $\ord  \  \partial_\mathfrak{l} \kappa'(\mathfrak{l})   = t_{\varphi,g}  - e_\mathfrak{l} $.
\item \label{it:basic-c} $\partial_\mathfrak{q} \kappa'(\mathfrak{l}) = 0$ for $\mathfrak{q}\nmid  \Delta\mathfrak{l}\mathfrak{p}$.
\item \label{it:basic-d} $\res_\mathfrak{q} \kappa'(\mathfrak{l}) \in \widehat{H}^1_\ord (K_{\infty, \mathfrak{q}}, T_{f,
n})\otimes_\varphi \mathcal{O}_f$ for $\mathfrak{q}|\Delta
\mathfrak{l}\mathfrak{p}$.
\end{enumerate}
\end{lem}
\begin{proof} Assertions (\ref{it:basic-a}) and (\ref{it:basic-b}) follow from the definition of
$\kappa'(\mathfrak{l})$ and the first reciprocity law. The latter
two assertions for $\mathfrak{q}\notin S$  follow from the fact
$\tilde{\kappa}'(\mathfrak{l})\in
\widehat{\mathrm{Sel}}^S_\Delta(K_\infty,
T_{f,r})\otimes_\varphi\mathcal{O}_f$.

We assume that $\mathfrak{q}\in S$ and $\mathfrak{q}\nmid
\Delta\mathfrak{l} $. As $\mathfrak{q}$ in $r$-admissible, by
Proposition \ref{prop-kill} (\ref{it:free}) we have that
$\widehat{H}^1(K_{\infty,\mathfrak{q}},
T_{f,r})\otimes_\varphi\mathcal{O}_f$ is free over
$\mathcal{O}_{f,r}[\Gamma]\otimes_\varphi\mathcal{O}_f$. Thus there
exists $s\in \widehat{H}^1_{\fin}(K_{\infty,\mathfrak{q}}, T_{f,r})$
such that $\omega^{e_\mathfrak{l}}s
=\res_\mathfrak{q}\kappa_\varphi(\mathfrak{l})$. This means that
$\omega^{e_{\mathfrak{l}}}(s- \res_\mathfrak{q}
\tilde{\kappa}'(\mathfrak{l}))=0$.  As $e_{\mathfrak{l}}\leq
t_{\varphi,g} \leq r-n$, from the freeness of
$\widehat{H}^1(K_{\infty,\mathfrak{q}},
T_{f,r})\otimes_\varphi\mathcal{O}_f$ we obtain that $s-
\res_\mathfrak{q} \tilde{\kappa}'(\mathfrak{l})\in \omega^{n}
\widehat{H}^1(K_{\infty,\mathfrak{q}},
T_{f,r})\otimes_\varphi\mathcal{O}_f$. Hence  the images of $s$ and
$\res_\mathfrak{q} \kappa'(\mathfrak{l})$ in
$\widehat{H}^1(K_{\infty,\mathfrak{q}},
T_{f,n})\otimes_\varphi\mathcal{O}_f$ coincide, which shows
(\ref{it:basic-c}) for $\mathfrak{q}\in S$.

By the same argument we can prove (\ref{it:basic-d}) for
$\mathfrak{q}\in S$.
\end{proof}

\begin{lem} $($\cite[Lemma 7.3.4]{Wang}$)$\label{lem:eta-l-zero}
Let $$\eta_\mathfrak{l}:
\widehat{H}^1_\sing(K_{\infty,\mathfrak{l}},
T_{f,n})\otimes_{\varphi}\mathcal{O}_f \rightarrow
\mathrm{Sel}_\Delta(K_\infty,
T_{f,n})^\vee\otimes_{\varphi}\mathcal{O}_f$$ be the map defined by
$$\eta_\mathfrak{l}(c)(x)=\langle c, \mathrm{res}_\mathfrak{l} (x)
\rangle_\mathfrak{l}$$ for $x\in \mathrm{Sel}_\Delta(K_\infty,
A_{f,n})[\ker(\varphi)]$ and $c\in
\widehat{H}^1_\sing(K_{\infty,\mathfrak{l}}, T_{f,n})$. Then
$\eta_\mathfrak{l}(\partial_\mathfrak{l}(\kappa'(\mathfrak{l})))=0$.
\end{lem}
\begin{proof} By global class field theory we have $\sum_\mathfrak{q}\langle \res_\mathfrak{q}\kappa'(\mathfrak{l}) ,
\res_\mathfrak{q}x\rangle_\mathfrak{q}=0$. When $\mathfrak{q}\neq
\mathfrak{l}$, both $\res_\mathfrak{q}\kappa'(\mathfrak{l})$ and
$\res_\mathfrak{q} x$ lie in the finite part or the ordinary part.
Thus by Proposition \ref{prop-kill} (\ref{it:fine-kill}) and
(\ref{it:ord-kill}), $ \langle
\res_\mathfrak{q}\kappa'(\mathfrak{l}) ,
\res_\mathfrak{q}x\rangle_\mathfrak{q}=0$ for $\mathfrak{q}\neq
\mathfrak{l}$. So $\langle
\partial_\mathfrak{l}\kappa'(\mathfrak{l}) ,
\res_\mathfrak{l}x\rangle_\mathfrak{l}=\langle
\res_\mathfrak{l}\kappa'(\mathfrak{l}) ,
\res_\mathfrak{l}x\rangle_\mathfrak{l}=0$.
\end{proof}

Choose an $r$-admissible prime ideal $\mathfrak{l}_1\notin S$ such
that
$$e_{\mathfrak{l}_1}=\min_{\tiny\begin{array}{l}\mathfrak{l}\notin
S\cup\{\mathfrak{q}_0\}: \\ r\text{-admissible}\end{array}}
e_{\mathfrak{l}},
$$ where $\mathfrak{q}_0$ is the prime chosen in Section
\ref{ss:first-second}.

\begin{lem}\label{lem:Wang} $($\cite[Lemma 7.3.5, 7.3.6]{Wang}$)$ \label{lem:two-cases}
\begin{enumerate}
\item\label{it:case-1} If $t_{\varphi,g}=0$, then $\Sel_\Delta (K_\infty, A_{f,n})^\vee  \otimes_\varphi
\mathcal{O}_f $ is trivial.
\item\label{it:case-2} If $t_{\varphi,g}>0$, then
$e_{\mathfrak{l}_1} < t_{\varphi,g}$.
\end{enumerate}
\end{lem}
\begin{proof} Assume that $\Sel_\Delta (K_\infty, A_{f,n})^\vee  \otimes_\varphi
\mathcal{O}_f \neq 0$. Then by Nakayama's lemma
$$ (\Sel_\Delta (K_\infty, A_{f,n})^\vee  \otimes_\varphi
\mathcal{O}_f)/(\omega) =  (\Sel_\Delta (K_\infty,
A_{f,n})[\mathfrak{m}])^\vee \otimes_\varphi \mathcal{O}_f $$ is
nonzero. Here, $\mathfrak{m}$ is the maximal ideal of
$\mathcal{O}_f[[\Gamma]]$. Let $x$ be a nonzero element in
$\Sel_\Delta (K_\infty, A_{f,n})[\mathfrak{m}].$ By Lemma
\ref{lem:control} (\ref{it:control-a}) we have
$$\Sel_\Delta (K_\infty, A_{f,n})[\mathfrak{m}]=\Sel_\Delta (K, A_{f,1}).$$
By Proposition \ref{prop:tool} there exists an $r$-admissible prime
$\mathfrak{l}\notin S$ such that $v_\mathfrak{l}(x)\neq 0$.

We show that $e_{\mathfrak{l}} < t_{\varphi,g}$ for this
$\mathfrak{l}$. Indeed, if $e_{\mathfrak{l}} = t_{\varphi,g}$, then
by Lemma \ref{lem:basic} (\ref{it:basic-b}),
$\partial_\mathfrak{l}\kappa'(\mathfrak{l})$ is indivisible and
hence is a generator of $H^1_\sing(K_{\infty,\mathfrak{l}},
T_{f,n})\otimes_\varphi\mathcal{O}_f$. By Proposition
\ref{prop-kill} (\ref{it:use-later}) and (\ref{it:free}) the image
of $\partial_\mathfrak{l}\kappa'(\mathfrak{l})$  in
$H^1_\sing(K_{\mathfrak{l}}, T_{f,1})$ is a generator. As
$\langle\cdot, \cdot\rangle_\mathfrak{l}$ induces a perfect pairing
between $\widehat{H}^1_\sing(K_{\mathfrak{l}}, T_{f,1})$ and
$H^1_\fin(K_{\mathfrak{l}}, A_{f,1})$, we have $\langle
\mathrm{Res}_\mathfrak{l} \kappa'(\mathfrak{l}) ,
\mathrm{Res}_\mathfrak{l}x \rangle_\mathfrak{l}\neq 0$. But this
implies that $\eta_\mathfrak{l}(\kappa'(\mathfrak{l}))\neq 0$, which
contradicts Lemma \ref{lem:eta-l-zero}.
\end{proof}

By Lemma \ref{lem:two-cases} (\ref{it:case-1}), (\ref{eq:key-form})
holds when $t_{\varphi,g}=0$. So we assume that $t_{\varphi, g}>0$.

\begin{lem} We have $$ e_{\mathfrak{l}_1} = \min_{\mathfrak{l}'} \ord
(v_{\mathfrak{l'}}\kappa_\varphi(\mathfrak{l}_1)), $$ where
$\mathfrak{l}'$ runs over all $r$-admissible prime ideals that do
not divide $\mathfrak{l}_1\mathfrak{q}_0\Delta$ and are not in $S$.
\end{lem}
\begin{proof} What we need to prove is that
$$ \min_{\mathfrak{l}'} \ord
(v_{\mathfrak{l'}}\tilde{\kappa}'(\mathfrak{l}_1)) =0 .$$ If $ \ord
(v_{\mathfrak{l'}}\tilde{\kappa}'(\mathfrak{l}_1))>0$ for each
$r$-admissible prime ideal $\mathfrak{l}'\notin S$ that do not
divide $\mathfrak{l}_1\mathfrak{q}_0\Delta$. Then $\kappa_1$, the
image of $\tilde{\kappa}'(\mathfrak{l}_1)$ in
$$ \widehat{\mathrm{Sel}}^S_\Delta(K_\infty,
T_{f,r})\otimes_\varphi \mathcal{O}_{f,1}=
\widehat{\mathrm{Sel}}^S_\Delta(K_\infty, T_{f,r})/\mathfrak{m}
\otimes_\varphi \mathcal{O}_f \hookrightarrow H^1(K, T_{f,1}),
$$ satisfies $\kappa_1\neq 0$, but
$v_{\mathfrak{l'}}(\kappa_1)=0$ for all $\mathfrak{l}'$ as above.
This contradicts Proposition \ref{prop:tool}.
\end{proof}

Let $\mathfrak{l}_2$ ($\mathfrak{l}_2 \nmid
\mathfrak{l}_1\mathfrak{q}_0\Delta$ and $\mathfrak{l}_2\notin S$) be
an $r$-admissible prime ideal such that $\ord\ v_{\mathfrak{l}_2}
(\kappa_\varphi(\mathfrak{l}_1)) = e_{\mathfrak{l}_1}$. In
particular, $v_{\mathfrak{l}_2}
(\kappa_{\mathscr{D}}(\mathfrak{l}_1)_m)\neq 0$. By the choice of
$\mathfrak{l}_2$ and the minimality of $e_{\mathfrak{l}_1}$ we have
\begin{equation}\label{eq:e1e2-1} \ord\ v_{\mathfrak{l}_2} (\kappa_\varphi(\mathfrak{l}_1)) =
e_{\mathfrak{l}_1} \leq e_{\mathfrak{l}_2}\leq \ord\
v_{\mathfrak{l}_1}(\kappa_\varphi(\mathfrak{l}_2)) .
\end{equation}
As
\begin{eqnarray*} \ord_\varphi v_{\mathfrak{l}_2}
(\kappa_{\mathscr{D}}(\mathfrak{l}_1)_m) + \ord_{\varphi}
\partial_{\mathfrak{l}_1} (\kappa_{\mathscr{D}}(\mathfrak{l}_1)_m)
&=&
\ord_\varphi v_{\mathfrak{l}_2}
(\kappa_{\mathscr{D}}(\mathfrak{l}_1)) + \ord_{\varphi}
\partial_{\mathfrak{l}_1} (\kappa_{\mathscr{D}}(\mathfrak{l}_1)) \\
&=& e_{\mathfrak{l}_1}+ t_{\varphi, g} < 2t_{\varphi,g} \leq
r,\end{eqnarray*} by Corollary \ref{cor:theta-2} we have
\begin{equation}\label{eq:e1e2-2} \ord\ v_{\mathfrak{l}_2}
(\kappa_\varphi (\mathfrak{l}_1)) =\ord_\varphi v_{\mathfrak{l}_2}
(\kappa_{\mathscr{D}}(\mathfrak{l}_1)_m) = \ord_\varphi
v_{\mathfrak{l}_1} (\kappa_{\mathscr{D}}(\mathfrak{l}_2)_m)  = \ord\
v_{\mathfrak{l}_1} (\kappa_\varphi (\mathfrak{l}_2)).
\end{equation} Combining (\ref{eq:e1e2-1}) and (\ref{eq:e1e2-2}) we
obtain
\begin{equation}\label{eq:e1e2-3} \ord\ v_{\mathfrak{l}_2} (\kappa_\varphi(\mathfrak{l}_1)) =
e_{\mathfrak{l}_1} = e_{\mathfrak{l}_2}= \ord\
v_{\mathfrak{l}_1}(\kappa_\varphi(\mathfrak{l}_2)) .
\end{equation}
It follows that
\begin{equation} \label{eq:comp-ord}
\ord\ v_{\mathfrak{l}_1}( \kappa'(\mathfrak{l}_2)) = \ord\
v_{\mathfrak{l}_2}( \kappa'(\mathfrak{l}_1)) =0.\end{equation}

Since $v_{\mathfrak{l}_2}
(\kappa_{\mathscr{D}}(\mathfrak{l}_1)_m)\neq 0$, by Proposition
\ref{thm:second} there exists an integer $r_0<r$ and an $(N,
r-r_0)$-admissible form $(\Delta \mathfrak{l}_1\mathfrak{l}_2 ,
g'')$ such that $$ r_0\leq \ord_\varphi\
v_{\mathfrak{l}_2}(\kappa(\mathfrak{l}_1)_m)=e_{\mathfrak{l}_1}$$
and
$$ v_{\mathfrak{l}_2}(\kappa_{\mathscr{D}}(\mathfrak{l}_1)_m)
= \omega^{r_0} \theta_m(g'') \in \mathcal{O}_{f,r}[\Gamma_m]
$$ up to multiplication by a unit of $\mathcal{O}_{f,r}[\Gamma_m]$.
It follows that
\begin{equation} r_0+t_{\varphi,g''} =  \ord_\varphi\ v_{\mathfrak{l}_2}(\kappa_{\mathscr{D}}(\mathfrak{l}_1)_m)
= e_{\mathfrak{l}_1}
 .\end{equation}

Let $S_{\mathfrak{l}_1,\mathfrak{l}_2}$ be the subgroup of
$\Sel_\Delta(K_\infty, A_{f,n})$ consisting of elements that are
locally trivial at the prime ideals dividing $\mathfrak{l}_1$ or
$\mathfrak{l}_2$. By the definition of Selmer groups, we have the
following two exact sequences
$$\xymatrix{ \widehat{H}^1_\sing(K_{\infty,\mathfrak{l}_1}, T_{f,n})\oplus \widehat{H}^1_\sing(K_{\infty,\mathfrak{l}_2}, T_{f,n}) \ar[r]^{\hskip 50pt \eta_s} &
\Sel_\Delta(K_\infty, A_{f,n})^\vee \ar[r] &
S^{\vee}_{\mathfrak{l}_1,\mathfrak{l}_2} \ar[r] & 0 }$$ and
$$\xymatrix{  \widehat{H}^1_\fin(K_{\infty,\mathfrak{l}_1}, T_{f,n}) \oplus \widehat{H}^1_\fin (K_{\infty,\mathfrak{l}_2}, T_{f,n}) \ar[r]^{\hskip 40pt \eta_f} &
\Sel_{\Delta \mathfrak{l}_1\mathfrak{l}_2}(K_\infty, A_{f,n})^\vee
\ar[r] & S^{\vee}_{\mathfrak{l}_1,\mathfrak{l}_2} \ar[r] & 0 , }$$
where $\eta_s$ and $\eta_f$ are induced by the local Tate pairing
$\langle \cdot, \cdot \rangle_{\mathfrak{l}_1} \oplus \langle \cdot,
\cdot \rangle_{\mathfrak{l}_2}$.

\begin{lem}\label{lem:kernel-eta-f}
We have $\eta_f^\varphi=0$.
\end{lem}
\begin{proof} From (\ref{eq:comp-ord}) we see $$\big(\widehat{H}^1_\fin(K_{\infty,\mathfrak{l}_1}, T_{f,n}) \oplus \widehat{H}^1_\fin (K_{\infty,\mathfrak{l}_2},
T_{f,n})\Big)\otimes_\varphi \mathcal{O}_f$$ is generated by
contains $(v_{\mathfrak{l}_1}(\kappa'(\mathfrak{l}_2)) , 0)$ and
$(0, v_{\mathfrak{l}_2}(\kappa'(\mathfrak{l}_1)) )$.

Let $s$ be in $\Sel_{\Delta \mathfrak{l}_1\mathfrak{l}_2}(K_\infty,
A_{f,n})$. By Lemma \ref{lem:basic} (\ref{it:basic-c}), for any
$\mathfrak{q}\nmid \Delta \mathfrak{l}_1\mathfrak{l}_2\mathfrak{p}$,
$$ \langle \kappa'(\mathfrak{l}_1) , s \rangle_{\mathfrak{q}} =\langle \partial_{\mathfrak{q}}\kappa'(\mathfrak{l}_1) , s \rangle_{\mathfrak{q}} =0.
$$ By Lemma \ref{lem:basic} (\ref{it:basic-d}), for any $\mathfrak{q}|
\Delta\mathfrak{l}_1\mathfrak{p}$, $$ \langle
\kappa'(\mathfrak{l}_1) , s \rangle_{\mathfrak{q}} =0.
$$ Thus by the global reciprocity law we have $ \langle v_{\mathfrak{l}_2}(\kappa'(\mathfrak{l}_1)) , s \rangle_{\mathfrak{l}_2}
= \langle \kappa'(\mathfrak{l}_1) , s \rangle_{\mathfrak{l}_2}=0. $
The same argument shows that $ \langle
v_{\mathfrak{l}_1}(\kappa'(\mathfrak{l}_2) ) , s
\rangle_{\mathfrak{l}_1} =0. $
\end{proof}

By Lemma \ref{lem:kernel-eta-f} we obtain
$$
S^{\vee}_{\mathfrak{l}_1,\mathfrak{l}_2} \otimes_\varphi
\mathcal{O}_f\cong \Sel_{\Delta
\mathfrak{l}_1\mathfrak{l}_2}(K_\infty, A_{f,n})^\vee
\otimes_\varphi \mathcal{O}_f .$$

As $$ r_0+2t_{\varphi,g''} \leq 2r_0+2t_{\varphi,g''}
=2e_{\mathfrak{l}_1} < 2t_{\varphi, g}\leq r,
$$ we have
$$ 2t_{\varphi,g''}\leq r-r_0.
$$ We also have
$$ n \leq r-t_{\varphi, g} < r- e_{\mathfrak{l}_1} = (r-r_0)-t_{\varphi,g''}.  $$

Hence, $(\Delta \mathfrak{l}_1\mathfrak{l}_2 , g'')$ is an $(N,
r-r_0)$-admissible form, $2t_{\varphi, g''} \leq r-r_0 $ and $n\leq
(r-r_0)-t_{\varphi, g''}$. As $t_{\varphi, g''}<t_{\varphi,g}$, by
the inductive assumption we have
$$\mathrm{length}_{\mathcal{O}_f}
S^{\vee}_{\mathfrak{l}_1,\mathfrak{l}_2} \otimes_\varphi
\mathcal{O}_f = \mathrm{length}_{\mathcal{O}_f} \Sel_{\Delta
\mathfrak{l}_1\mathfrak{l}_2}(K_\infty, A_{f,n})^\vee
\otimes_\varphi \mathcal{O}_f \leq 2t_{\varphi,g''}.$$ By Lemma
\ref{lem:eta-l-zero}, $\eta_s^\varphi$ factors through the quotient
$$
\mathcal{O}_f/(\partial_{\mathfrak{l}_1}\kappa'(\mathfrak{l}_1))\oplus
\mathcal{O}_f/(\partial_{\mathfrak{l}_2}\kappa'(\mathfrak{l}_2)). $$
Thus
\begin{eqnarray*} \mathrm{length}_{\mathcal{O}_f} \Sel_\Delta(K_\infty, A_{f,n})^\vee \otimes_\varphi \mathcal{O}_f & \leq & \ord \
\partial_{\mathfrak{l}_1}(\kappa'(\mathfrak{l}_1)) + \ord \
\partial_{\mathfrak{l}_2}(\kappa'(\mathfrak{l}_2)) + \mathrm{length}_{\mathcal{O}_f}
S^{\vee}_{\mathfrak{l}_1,\mathfrak{l}_2} \\ & \leq &
(t_{\varphi,g}-e_{\mathfrak{l}_1}) +
(t_{\varphi,g}-e_{\mathfrak{l}_2}) + 2 t_{\varphi, g''}\\  &=&
2t_{\varphi, g}-2r_0 \leq 2t_{\varphi, g}.
\end{eqnarray*} This finishes the inductive argument of the proof of Theorem \ref{thm:auxi} in the case of
$(\mathfrak{n}^+\text{-}\mathrm{min})$. \qed \vskip 15pt

\noindent{\it Proof of Theorem \ref{thm:main} in the case of
$(\mathfrak{n}^+\text{-}\mathrm{min})$.} Let $\varphi:
\mathcal{O}_f[[\Gamma]]\rightarrow \mathcal{O}$ be a homomorphism as
at the beginning of this section. Enlarging $\mathcal{O}_f$ if
necessary we may assume that $\mathcal{O}=\mathcal{O}_f$.

If $\varphi(L_p(K_\infty,f))=0$, then $$\varphi(L_p(K_\infty,f))\in
\mathrm{Fitt}_{\mathcal{O}}(\mathrm{Sel}_{\mathfrak{n}^-}(K_\infty,
A_{f})^\vee\otimes_{\varphi}\mathcal{O}) .$$ Therefore we may assume
that $\varphi(L_p(K_\infty,f))\neq 0$. Choose $t^*$ larger than
$\mathrm{ord} \: \varphi(L_p(K_\infty,f))$. For each positive
integer $n$ consider the $(n+t^*)$-admissible form
$\mathscr{D}_{n+t^*}=(\mathfrak{n}^-, f^\dagger_{n+t^*})$ provided
by Proposition \ref{prop:exist-adm}. By Theorem \ref{thm:auxi} we
obtain that
$$\varphi(L_p(K_\infty,f))=\varphi(\theta_{f^\dagger_{n+t^*}})^2
\hskip 3pt (\mathrm{mod} \ \omega^{n+t^*})$$ belongs to
$\mathrm{Fitt}_{\mathcal{O}}(\mathrm{Sel}_{\mathfrak{n}^-}(K_\infty,
A_{f,n})^\vee\otimes_{\varphi}\mathcal{O})$ when $n\geq t^*$.
Therefore $\varphi(L_p(K_\infty,f))$ belongs to
$$\mathrm{Fitt}_{\mathcal{O}}(\mathrm{Sel}_{\mathfrak{n}^-}(K_\infty,
A_{f})^\vee \otimes_\varphi \mathcal{O} )=\bigcap_{n\geq t^*}
\mathrm{Fitt}_{\mathcal{O}}(\mathrm{Sel}_{\mathfrak{n}^-}(K_\infty,
A_{f,n})^\vee \otimes_\varphi \mathcal{O}).$$ By \cite[Lemma
6.11]{CH15} we have $$L_p(K_\infty, f)\in
\mathrm{Fitt}_{\mathcal{O}}(\mathrm{Sel}_{\mathfrak{n}^-}(K_\infty,
A_{f})^\vee). $$ As $L_p(K_\infty, f)\neq 0$,
$\mathrm{Sel}_{\mathfrak{n}^-}(K_\infty, A_{f})$ is
$\Lambda$-cotorsion.   \qed

\vskip 10pt

We relax the condition $(\mathfrak{n}^+\text{-}\mathrm{min})$ to
$(\mathfrak{n}^+\text{-}\mathrm{DT})$.

\begin{lem}\label{Jarvis} There exists a Hilbert modular form $f'$ congruence to $f$
modulo $\omega$ that satisfies $(\mathrm{CR}^+)$, $(\mathrm{PO})$
and $(\mathfrak{n}^+\text{-}\mathrm{min})$.
\end{lem}
\begin{proof} We need to show that, if $\mathfrak{l}| \frac{\mathfrak{n}}{\mathfrak{n}_{\bar{\rho}}}$, then there exists a Hilbert modular form $f'$ of level dividing $\frac{\mathfrak{n}}{\mathfrak{l}}$ congruence to $f$. In the case $\pi_\mathfrak{l}$ is special or supercuspidal, this follows directly from
Jarvis's level lowering result \cite[Theorem
0.1]{Jav}. Note that our condition $(\mathfrak{n}^+\text{-}\mathrm{DT})$ ensures that $f$ satisfies conditions of \cite[Theorem
0.1]{Jav}.

Now let $\pi_\mathfrak{l}=\mathrm{Ind}^{\mathrm{GL}_2(F_\mathfrak{l})}_{B}(\chi\otimes \chi^{-1})$ be a principal series representation, where $\chi$ is a character of $F^\times_\mathfrak{l}$, and $B$ is the Borel subgroup of $\mathrm{GL}_2(F_\mathfrak{l})$ consisting of upper-triangular invertible matrices. When the conductor $\mathfrak{n}_\chi$ of $\chi$ is $\mathfrak{l}$, $f$ again satisfies the condition of \cite[Theorem 0.1]{Jav}, and so we can apply Jarvis's result. It remains to show that, either if $\mathfrak{n}_\chi$ is $\mathcal{O}_{F_\mathfrak{l}}$, or if $\mathfrak{n}_\chi$ is divided by $\mathfrak{l}^2$, then $\mathfrak{l}\nmid \frac{\mathfrak{n}}{\mathfrak{n}_{\bar{\rho}}}$. In the former case there is nothing to prove. In the latter case, observe that the conductor of $\bar{\chi}=\chi  (\mathrm{mod} \:  \omega)$ is equal to that of $\chi$. It follows that the conductor of $\bar{\rho}_{f,\mathfrak{l}}$ is equal to that of $\rho_{f,\mathfrak{l}}$, since $\rho_{f,\mathfrak{l}}\cong \chi\oplus \chi^{-1}$ when it is restricted to the inert subgroup of $G_{F_\mathfrak{l}}$ \cite{Tay}.
\end{proof}

\begin{prop} Assume that $(\mathrm{CR}^+)$, $(\mathrm{PO})$  and
$(\mathfrak{n}^+\text{-}\mathrm{DT})$ hold. Then
$\mathrm{Sel}^{\mathfrak{n}^+}_{\mathfrak{n}^-}(K_\infty, A_f)$ is
$\Lambda$-cotorsion and
$\mathrm{Sel}^{\mathfrak{n}^+}_{\mathfrak{n}^-}(K_\infty, A_f)^\vee$
has vanishing $\mu$-invariant.
\end{prop}
\begin{proof} Let $f'$ be as in Lemma \ref{Jarvis}. Then Theorem
\ref{thm:main} holds for $f'$. Combining this with Proposition
\ref{prop-Hung} we obtain that
$\mathrm{Sel}^{\mathfrak{n}^+}_{\mathfrak{n}^-}(K_\infty,
A_{f'})^\vee$ has vanishing $\mu$-invariant. In other words,
$\mathrm{Sel}^{\mathfrak{n}^+}_{\mathfrak{n}^-}(K_\infty,
A_{f'})[\omega]$ is finite.

By Lemma \ref{lem:control} (\ref{it:control-b}), taking
$S=\mathfrak{n}^+$, we get
$$\mathrm{Sel}^{\mathfrak{n}^+}_{\mathfrak{n}^-}(K_\infty,
A_f)[\omega]=\mathrm{Sel}^{\mathfrak{n}^+}_{\mathfrak{n}^-}(K_\infty,
A_{f,1})=\mathrm{Sel}^{\mathfrak{n}^+}_{\mathfrak{n}^-}(K_\infty,
A_{f',1})=\mathrm{Sel}^{\mathfrak{n}^+}_{\mathfrak{n}^-}(K_\infty,
A_{f'})[\omega].$$  Thus
$\mathrm{Sel}^{\mathfrak{n}^+}_{\mathfrak{n}^-}(K_\infty,
A_f)[\omega]$ is finite, and
$\mathrm{Sel}^{\mathfrak{n}^+}_{\mathfrak{n}^-}(K_\infty, A_f)^\vee$
has vanishing $\mu$-invariant.
\end{proof}

Since $\mathrm{Sel}^{\mathfrak{n}^+}_{\mathfrak{n}^-}(K_\infty,
A_f)^\vee$ has vanishing $\mu$-invariant, by Theorem
\ref{prop-cond-free}, $\widehat{\mathrm{Sel}}_{\Delta}(K_\infty,
T_{f,N})$ is free over $\Lambda/\omega^N$. Now repeating the
argument for the case of $(\mathfrak{n}^+\text{-}\mathrm{min})$  we
finish the proof of Theorem \ref{thm:main}. The only place we need
to revise the argument is the proof of Lemma \ref{lem:Wang}. Assume
that $\Sel_\Delta (K_\infty, A_{f,n})$ is nonzero. In general, we
may not have $\Sel_\Delta (K_\infty,
A_{f,n})[\mathfrak{m}]=\Sel_\Delta (K, A_{f,1})$ now. But by Lemma
\ref{lem:control} (\ref{it:control-b}) we have $$\Sel_\Delta
(K_\infty, A_{f,n})[\mathfrak{m}]\subseteq
\Sel^{\mathfrak{n}^+}_\Delta (K_\infty,
A_{f,n})[\mathfrak{m}]=\Sel^{\mathfrak{n}^+}_\Delta (K, A_{f,1}).$$
Consider the nonzero element $x$ in $\Sel_\Delta (K_\infty,
A_{f,n})[\mathfrak{m}]$ as an element in
$\Sel^{\mathfrak{n}^+}_\Delta (K, A_{f,1})$. If
$e_{\mathfrak{l}}=t_{\varphi, g}$, we again obtain $\langle
\mathrm{Res}_\mathfrak{l} \kappa'(\mathfrak{l}) ,
\mathrm{Res}_\mathfrak{l}x \rangle_\mathfrak{l}\neq 0$ and
$\eta_\mathfrak{l}(\kappa'(\mathfrak{l}))\neq 0$, contradicting
Lemma \ref{lem:eta-l-zero}. \qed \vskip 10pt

\section{Examples} \label{sec:example}

We construct examples of CM type. Assume $p\geq 7$, and fix an even
integer $k$ such that $4\leq k<p-1$.

Let $L$ be a totally imaginary quadratic extension of $F$ that is
disjoint from $K$.

We fix a CM type $\Sigma_L$ of $L$, i.e. for each real place
$\sigma$ of $F$, there exists exactly one element of $\Sigma_L$ that
restricts to $\sigma$. We again use $\sigma$ to denote this element,
and use $\bar{\sigma}$ to denote the other embedding of $L$ above
$\sigma$. We fix an isomorphism of fields $\jmath:
\mathbb{C}\xrightarrow{\sim} \mathbb{C}_p$. Then $\jmath$ and
$\Sigma_K$ determine a set $\Sigma'_L$ of $[F:\mathbb{Q}]$
embeddings $L\hookrightarrow\mathbb{C}_p$.

Let $\Omega_p$ be the maximal ideal of $\mathcal{O}_{\mathbb{C}_p}$.
For each $\iota\in \Sigma_L$ put $\mathfrak{p}_\iota=\iota(F)\cap
\jmath^{-1}\Omega_p$, $\mathfrak{P}_\iota=\iota(L)\cap
\jmath^{-1}\Omega_p$ and
$\bar{\mathfrak{P}}_\iota=\bar{\iota}(L)\cap \jmath^{-1}\Omega_p$.
Note that, for a prime $\mathfrak{q}$ above $p$, the set of
$\iota\in \Sigma_L$ such that $\mathfrak{p}_\iota=\mathfrak{q}$ has
cardinal number $[F_\mathfrak{q}:\mathbb{Q}_p]$.

%

% Then $\jmath^{-1}\Omega_p \cap \sigma(Q)$ is a prime ideal
% $\mathfrak{p}$ of $\mathcal{O}_K$ above $p$.

Let $\chi: \mathrm{Gal}(\overline{L}/L)\rightarrow
\mathcal{O}_{\mathbb{C}_p}^\times$ be a $p$-adic continuous
character. Let $\mathrm{rec}_L: L^\times\backslash
\mathbb{A}_L^\times \rightarrow \mathrm{Gal}(\overline{L}/L)$ be the
geometrically normalized reciprocity map. By $\mathrm{rec}_L$ we may
view $\chi$ as a character of $\mathbb{A}_L^\times$ that is trivial
on $L^\times$. We assume that $\chi$ satisfies the following
conditions.

(ram)  Let $\bar{\chi}$ denote the reduction of $\chi$ modulo
$\Omega_p$ (the maximal ideal of $\mathcal{O}_{\mathbb{C}_p}$). Then
for each finite place $v$ of $F$ that is coprime to $p$, $\chi_v$
and $\bar{\chi}_v$ have the same Artin conductors
$\mathfrak{N}_{\chi_v}$.

(weight) At each place $\mathfrak{p}$ of $F$ above $p$, $\chi$ is of
type $(k/2, -k/2)$, which is explained as below.

In the case when $\mathfrak{p}$ is inert or ramified in $L$, i.e.
$\mathfrak{p}\mathcal{O}_L=\mathfrak{P}$ or $\mathfrak{P}^2$, we
have
$\chi_{\mathfrak{p}}(a)=(\frac{a}{\bar{a}})^{[F_\mathfrak{p}:\mathbb{Q}_p]k/2}$
on $\mathcal{O}_{L_\mathfrak{P}}^\times$.

In the case when $\mathfrak{p}$ is split in $L$, i.e.
$\mathfrak{p}\mathcal{O}_L=\mathfrak{P}\overline{\mathfrak{P}}$
($\mathfrak{P}\in \Sigma'_L$) and $F_\mathfrak{p}\cong
L_\mathfrak{P}\cong L_{\overline{\mathfrak{P}}}$, we have
$\chi_{\mathfrak{P}}(a)=a^{[F_\mathfrak{p}:\mathbb{Q}_p]k/2}$ and
$\chi_{\overline{\mathfrak{P}}}(a)=a^{-[F_\mathfrak{p}:\mathbb{Q}_p]k/2}$
on $\mathcal{O}_{F_\mathfrak{p}}^\times$.

(triv) The restriction of $\chi$ to $\mathbb{A}_{F}^\times$
coincides with the character $\tau_{L/F}$ attached to the extension
$L/F$.

(irred) $\bar{\chi}^2\neq 1$.

(dis) The split field $F(\bar{\chi})$ of $\bar{\chi}$ is disjoint
from $F(\sqrt{p^*})$ where $p^*=(-1)^{\frac{p-1}{2}}p$.

We attach to $\chi$ a complex Hecke character $\chi_\infty$ of
$\mathbb{A}_L^\times$ by
$$\chi_\infty (a)= \jmath^{-1}\Big(\chi(a)\cdot \prod_{\mathfrak{P}\in \Sigma'_K}a_{\mathfrak{P}_\sigma}^{-k/2}
\prod_{\sigma\in \overline{\Sigma}_K}
{a_{\overline{\mathfrak{P}}_\sigma}}^{k/2}\Big) \cdot \prod_{\infty
\in \Sigma_K} a_\infty^{k/2} \prod_{\infty \in
\overline{\Sigma}_K}a_\infty^{-k/2}.
$$

Let $\mathfrak{N}_\chi$ be the conductor of $\chi_\infty$.  By the
condition (weight), $\chi_\infty$ is unramified at primes above $p$,
and thus $p$ is coprime to $\mathfrak{N}_\chi$. We have
$$\mathfrak{N}_\chi=\prod_{v: \text{ finite and coprime to } p}
\mathfrak{N}_{\chi_v}.$$ Write
$\mathfrak{N}_\chi=\mathfrak{N}_\chi^+\mathfrak{N}_\chi^-$, where
$\mathfrak{N}_\chi^+$ consists of primes that lie above primes of
$F$ split in $K$, and $\mathfrak{N}_\chi^+$ consists of primes that
lie above primes of $F$ inert or ramified in $K$.

(free)  $\mathfrak{N}^-_\chi$ is coprime to
$\mathfrak{D}_{L/F}$; if $w|\mathfrak{N}^-_\chi$, then
$w||\mathfrak{N}^-_\chi$, $w$ lies above a prime split $v=w\bar{w}$
in $L$, and $\bar{w}\nmid \mathfrak{N}^-_\chi$.

Put
$$\mathfrak{n}=N_{L/F}(\mathcal{D}_{L/F}\mathfrak{N}_\chi), \hskip
10pt  \mathfrak{n}^+=N_{L/F}(\mathfrak{N}_\chi^+), \hskip 10pt
\mathfrak{n}^-=N_{L/F}(\mathcal{D}_{L/F}\mathfrak{N}^-_\chi) .$$
Then $\mathfrak{n}=\mathfrak{n}^+\mathfrak{n}^-$. By (free),
$\mathfrak{n}^-$ is square-free.

\begin{prop}\label{prop:CM} If $\chi$ satisfies $(\mathrm{weight})$ and $(\mathrm{triv})$, then there exists a newform $f_\chi\in S_k(\mathfrak{n},1)$ such that
$$ L(s,f_\chi) = L(s+\frac{1-k}{2},\chi_\infty) . $$
\end{prop}
\begin{proof} In the case of $F=\BQ$ this follows from \cite[Chapter 7]{Li}. In the
general case we deduce our assertion from Jacquect and Langlands's
result \cite[Proposition 12.1 ]{JL}.

By loc. cit., there exists a cuspidal automorphic representation
$\pi_F$ of $\mathrm{GL}(2,\BA_F)$ that is automorphically induced
from $\chi_\infty$ so that $L(s, \pi_F)=L(s, \chi_\infty)$.

At each real place $v$, the representation of the Weil group
$W_{F_v}$ corresponding to $\pi_{F,v}$ is induced from the character
$a\mapsto (\frac{a}{\bar{a}})^{k/2}$ of $W_{\mathbb{C}}$. By the
local Langlands correspondence for $\mathrm{GL}_2(\BR)$, our
$\pi_{F,v}$ is the discrete series of weight $k$ and trivial central
character.

Let $v$ be a finite place of $F$. When $v$ splits in $L$, say
$v=w\bar{w}$, $L_v=L_w\oplus L_{\bar{w}}\cong F_v^{\oplus 2}$ and
the two characters $\chi_{\infty,w}$ and $\chi_{\infty, \bar{w}}$
satisfies $\chi_{\infty,w}\chi_{\infty, \bar{w}}=1$. In this case,
$\pi_{F,v}$ is isomorphic to $\pi(\chi_{\infty,w},
\chi_{\infty,\bar{w}})$ and thus has trivial central character. We
use $\mathfrak{N}(\chi_w)$ to denote the conductor of $\chi_w$. Then
by \cite{Cassel} (see the proof of the second lemma in
\cite{Cassel}) the conductor of $\pi_v\cong \pi(\chi_{\infty,w},
\chi_{\infty,\bar{w}})$ is
$\mathfrak{N}(\chi_w)\mathfrak{N}(\chi_{\bar{w}})$.

When $v$ is inert or ramified in $L$, via local Langlands
correspondence $\pi_{F,v}$ is attached to
$\mathrm{Ind}_{L_v/F_v}\chi_{\infty, v}$ the representation of
$W_{F_v}$ induced from $\chi_{\infty,v}$. The central character of
$\mathrm{Ind}_{L_v/F_v}\chi_{\infty, v}$, that is $ \chi_{\infty,
v}|_{F_v}\cdot\tau_{L_v/F_v}$, is trivial, and thus so is the
central character of $\pi_{F,v}$. It is shown in \cite{Hen} that,
the conductor of $\pi_{F,v}$ is equal to the Artin conductor of
$\mathrm{Ind}_{L_v/F_v}\chi_{\infty, v}$, which is
$N_{L_v/F_v}(\mathcal{D}_{L_v/F_v}\mathfrak{N}(\chi_{\infty,v}))$ by
Corollary of Proposition $4$ in \cite[Chapter VI, \S 2]{Serre-LF}.

Hence, $\pi_F$ has trivial central character and the conductor of
$\pi_F$ is $\mathfrak{n}$. Therefore, in $\pi_F$ there exists a
(non-zero holomorphic) cuspidal newform $f_\chi$ of weight $k$,
level $\mathfrak{n}$ and trivial central character. By the theory of
automorphic representations one has $L(s,
f_\chi)=L(s+\frac{1-k}{2},\pi_F)$. This finishes our proof.
\end{proof}

\begin{prop} \label{prop:CM-exam}
Suppose that $\chi$ satisfies the conditions $(\mathrm{ram})$,
 $(\mathrm{weight})$, $(\mathrm{triv})$,  $(\mathrm{irred})$, $(\mathrm{dis})$ and $(\mathrm{free})$ ,
 and that
$[F:\BQ]$ has the same parity as the number of prime factors of
$\mathfrak{n}^-$. Then $f_\chi$ satisfies $(\mathrm{CR}^+)$,
$(\mathfrak{n}^+$-$\mathrm{DT})$ and $(\mathrm{PO})$.
\end{prop}
\begin{proof}
As $p\geq 7$ and $k\geq 4$, ($\mathrm{CR}^{+}$-1) and (PO) hold
automatically.

The Galois representation $\rho_{f_\chi}$ is isomorphic to
$\mathrm{Ind}_{L/F} \chi$, and thus the residue representation
$\bar{\rho}_{f_\chi}$ is isomorphic to
$\mathrm{Ind}_{L/F}\bar{\chi}$. Let $c$ be in $G_F \backslash G_L$.
As the restriction of $\bar{\chi}$ to $\BA_F^\times$ is trivial,
$\bar{\chi}^c=\bar{\chi}^{-1}$. By (irred), $\bar{\chi}^c\neq
\bar{\chi}$, so $\bar{\rho}_{f_\chi}$ is irreducible.

By (ram) $\bar{\rho}_{f_\chi}$ and $\rho_{f_\chi}$ have the same
conductor, thus is equal to $\mathfrak{n}$, which ensures
($\mathrm{CR}^{+}$-3), ($\mathrm{CR}^{+}$-4) and
$(\mathfrak{n}^+$-$\mathrm{DT})$.

By (dis) the restriction of $\rho_{f_\chi}$ to $G_{F(\sqrt{p^*})}$
is again irreducible, which implies ($\mathrm{CR}^{+}$-2).
\end{proof}

\begin{rem} Obviously $\rho_{f_\chi}\cong \mathrm{Ind}_{L/F} \chi$
does not satisfy Manning and Shotton's condition in \cite{ManSho}.
Indeed, the image of the residue representation
$\bar{\rho}_{f_\chi}\cong \mathrm{Ind}_{L/F} \bar{\chi}$ is a
dihedral group and contains no subgroup isomorphic to
$\mathrm{SL}_2(\mathbb{F}_p)$. So, our main theorem can apply to
these $f_\chi$ in Proposition \ref{prop:CM-exam}, but Manning and
Shotton's result can not.
\end{rem}


\begin{thebibliography}{99}

% \bibitem{BD94} M. Bertolini, H. Darmon: {\it Derived heights and generalized Mazur-
% regulators.} Duke Math. J. 76 (1994) 75-111.

\bibitem{BD05} M. Bertolini, H. Darmon: {\it Iwasawa's main conjecture for elliptic curves over anticyclotomic
$\BZ_p$-extensions.} Ann. Math. (2) 162 (2005), 1-64.

\bibitem{BLR} S. Bosch, W. L\"utkebohmert, M. Raynaud: {\it N\'eron
model.} A Series of Modern Surveys in Mathematics, vol. 21,
Springer-Verlag, Berlin 1990.

\bibitem{Carayol} H. Carayol: {\it Sur les repr\'esentations $l$-adiques associ\'ees aux formes modulaires de
Hilbert.} Ann. Sci. \'Ec. Norm. Sup. (4) 19 (1986), 409-468.

\bibitem{Cassel}  W. Casselman: {\it On some results of Atkin and
Lehner}. Math. Ann. 201 (1973), 301-314.

\bibitem{Cheng} C. Cheng: {\it Ihara's lemma for Shimura curves}, private manuscript, 2011.

\bibitem{CH15} M. Chida, M.-L. Hsieh: {\it On the Iwasawa main conjecture for modular
forms.} Compos. Math. 151 (2015), 863-897.

\bibitem{CH16} M. Chida, M.-L. Hsieh: {\it Special values of anticyclotomic $L$-functions for modular forms.} J. Reine Angew. Math. 741 (2018), 87-131.

\bibitem{DT} F. Diamond, R. Taylor: {\it Nonoptimal levels of mod $l$ modular
representations.} Invent. math. 115 (1994), 435-462.

% \bibitem{DN67} K. Doi, H. Naganuma: {\it On the algebraic curves uniformized by arithmetical automorphic
% functions.} Ann. Math. 86 (1967), 449-460.

% \bibitem{DN69} K. Doi, H. Naganuma: {\it On the functional equation of certain Dirichlet
% series.} Invent. Math. 9 (1969), 1-14.

\bibitem{Fou} O. Fouquet: {\it Dihedral Iwasawa theory of nearly ordinary
quaternionic automorphic forms}. Compos. Math. 149 (2013), 356-416.

% \bibitem{GV} R. Greenberg, V. Vatsal: {\it On the Iwasawa invariants of elliptic curves.} Invent. Math. 142 (2000) 17-63.

\bibitem{Hen} G. Henniart: {\it Une preuve simple des conjectures de langlands pour $\mathrm{GL}(n)$ sur un corps
$p$-adique.} Invent. Math. 139 (2000), 439-455.

\bibitem{Hun} P. Hung: {\it On the vanishing mod $l$ of central $L$-values with anticyclotomic
twists for Hilbert modular forms.} J. Number Theory 173 (2017),
170-209.
%PhD thesis, National Taiwan University, 2014.

\bibitem{JL} H. Jacquet, P. Langlands: {\it Automorphic forms on
$\mathrm{GL}(2)$.}  Springer Lecture Notes in Math. vol. 114, 1970.

\bibitem{Jav} F. Jarvis: {\it Level lowering for modular mod $l$ representations over totally real
fields.} Math. Ann. 313 (1999), 141-160.

% \bibitem{KW-1} C. Khare, J.-P. Wintenberger: {\it Serre's modularity conjecture. I.} Invent. Math. 178 (2009), 485-504.

% \bibitem{KW-2} C. Khare, J.-P. Wintenberger: {\it Serre's modularity conjecture. II.} Invent. Math. 178 (2009), 505-586.

\bibitem{Poll-West} C.-H. Kim, R. Pollack, T. Weston:  {\it On the freeness of anticyclotomic Selmer groups of modular forms}.
Intern. J. Number Theory 13 (2017), 1443-1455.

\bibitem{Li} W.-C. Winnie Li: {\it Number theory with applications.} Series on
University Mathematics, 7. World Scientific Publishing Co., Inc.,
River Edge, NJ, 1996.

\bibitem{Longo} M. Longo: {\it Anticyclotomic Iwasawa main conjecture for Hilbert modular
forms}. Comment. Math. Helv. 87 (2012), 303-353.

\bibitem{ManSho} J. Manning, J. Shotton: {\it Ihara's lemma for Shimura curves over totally real fields via
patching.} Preprint 2019, arXiv:1907.06043v1.

\bibitem{Nek} J. Nekovar: {\it Level raising and Selmer groups for Hilbert modular forms of weight two}, Canad. J.
Math. 64 (2012), 588-668.

\bibitem{PW} R. Pollack, T. Weston: {\it On anticyclotomic $\mu$-invariants of modular
forms.} Compos. Math., 147 (2011), 1353-1381.

\bibitem{Serre-LF} J.-P. Serre: {\it Local field.} Graduate Texts in
Math. vol. 67, Springer-Verlag, New York, 1979.

% \bibitem{Serre} J.-P. Serre: {\it Sur les repr\'esentations modulaires de degr\'e $2$ de
% $\mathrm{Gal}(\overline{\mathbf{Q}}/\mathbf{Q})$.} Duke Math. J. 54 (1987), 179-230.

\bibitem{Shin} T. Shintani: {\it On lifting of holomorphic cusp
forms}. Proc. Symp. Pure Math. XXXIII, Part II, Providence (1979),
79-110.

\bibitem{Tay} R. Taylor: {\it On Galois representations associated to Hilbert modular
forms.} Invent. Math. 98 (1989), 265-280.

\bibitem{Wang} H. Wang: {\it Anticyclotomic Iwasawa Theory for Hilbert modular forms.} PhD thesis, The Pennsylvania State University, 2015.

\bibitem{Wiles} A. Wiles: {\it On ordinary $\lambda$-adic representatios associated to modular
forms.} Invent. Math. 94 (1988), 529-573.

\bibitem{Zhang} S. Zhang: {\it Gross-Zagier formula for
$\mathrm{GL}_2$.} Asian J. Math. 5 (2001), 183-290.

\end{thebibliography}
\end{document}